\documentclass[11pt,leqno]{article}

\usepackage{mathrsfs}

\usepackage{amsmath}
\usepackage{amssymb}
\usepackage{amsthm}

\usepackage{comment}

\usepackage{color}

\usepackage{accents}
\usepackage[dvipdfmx]{graphicx}
\usepackage{graphicx}

\newcommand\munderbar[1]{%
 \underaccent{\bar}{#1}}
\newtheorem{lem}{Lemma}[section]
\newtheorem{thm}{Theorem}[section]
\newtheorem{pro}{Proposition}[section]
\newtheorem{cor}{Corollary}[section]
\numberwithin{equation}{section}
\newtheorem{definition}{Definition}[section]

\allowdisplaybreaks[2]
\title{A completion of our earlier work on the Cauchy problem for non-effectively hyperbolic operators}

\author{ Tatsuo Nishitani\thanks{Department of Mathematics, Osaka University, 
Machikaneyama 1-1, Toyonaka, 560-0043, Osaka, Japan}}

 \date{}
\begin{document}

\maketitle

\begin{abstract}{For hyperbolic differential operators $P$ with non-effectively hyperbolic double characteristics, we study the relationship between the Gevrey well-posedness threshold for strong well-posedness and the associated Hamilton map and flow. In our previous work, we showed that if the Hamilton map has a Jordan block of size $4$ on the double characteristic manifold $\Sigma$ of codimension $3$, then the Cauchy problem for $P$ is well-posed in the Gevrey class $1<s<3$ for all lower-order terms, and that this result is optimal. Moreover, if there are no bicharacteristics tangent to $\Sigma$, then the Cauchy problem is well-posed in the Gevrey class $1<s<4$ for all lower-order terms, and this result is also optimal. 
In the present paper, we remove the restriction on the codimension of $\Sigma$,  thereby completing the result.}
\end{abstract}

\maketitle

\def\R{{\mathbb R}}
\def\C{{\mathbb C}}
\def\lr#1{\langle{#1}\rangle}
\def\al{\alpha}
\def\be{\beta}
\def\ep{\epsilon}
\def\ga{\gamma}
\def\si{\sigma}
\def\la{\lambda}
\def\te{\theta}
\def\La{\Lambda}
\def\dif{\partial}
\def\N{{\mathbb N}}
 \def\xiga{\langle{\xi}\rangle_{\!\gamma}}
\def\bg{{\bar g}}

\def\tka{{\munderbar t}}

\def\op#1{{\rm op}(#1)}
\def\sieq{\overset{\Sigma'}{=}}

\def\ep{\epsilon}
\def\varep{\varepsilon}
\def\bro{\bar\rho}

\section{Introduction}

Let
\[
P=-D_t^2+\sum_{|\al|+j\leq 2,j<2}
a_{j,\al}(t, x)D_x^{\al}D_t^j=P_2+P_1+P_0
\]
be a second-order differential operator with coefficients  $a_{j,\al}(t, x)$ which are real analytic or Gevrey class $s$ in $x$ ($s>1$ is close to $1$),  
defined near the origin
 of ${\mathbb R}^{n+1}$. Denote the principal 
 symbol by $p(t, x,\tau, \xi)$, hyperbolic 
 with respect to the $t$ direction, where 
 $x=(x_1,...,x_n)$, $\xi
 =(\xi_1,...,\xi_n)$. For notational convenience and clarity, we frequently write $t$, $\tau$, $D_t$, and $D_x$ as $x_0$, $\xi_0$, $D_0$, and $D$. Let $\rho\in T^*{\mathbb R}^{n+1}$ be a critical point of $p$. It is well known that the Hamilton map $F_p(\rho)=dH_p(\rho)$ either has two nonzero real eigenvalues--this is the effectively hyperbolic case, which has been extensively studied and is by now well understood  (see, for example \cite{Ni:JPDO} for a recent review and a simplification of the arguments)--or has a purely imaginary spectrum. 
 
 In the latter case, $C^{\infty}$ well-posedness of the Cauchy problem requires the so-called Ivrii-Petkov-H\"ormander condition (IPH condition for short; see \cite{IP}, \cite{Ho1}), namely that is the subprincipal symbol of $P$ lies between $-\Sigma\mu_j$ and $\Sigma\mu_j$, where $i\mu_j$ are the positive imaginary eigenvalues of $F_p(\rho)$. If the IHP condition is violated, it is known that the Cauchy problem fails to be well posed in the Gevrey class $s$ for sufficiently large $s$. 
 
This naturally leads us to ask for the best possible value of $s$, that is, the largest Gevrey class in which the Cauchy problem remains well posed. 
However, since the IPH condition involves both the subprincipal symbol and the eigenvalues of the Hamilton map--hence depends on both the principal part and the lower-order terms--we simplify the problem by instead determining the optimal (largest) Gevrey class for which the Cauchy problem is well-posed for all lower-order terms. This formulation clearly yields a problem that depends only on the principal part. 

Hereafter, unless otherwise stated, we consider only functions and symbols with Gevrey regularity of order $s$, enough close to $1$. For the precise definition of these notions, see the Appendix. As proved in \cite{Ho1}, the localization $p_{\rho}$, defined by 
\[
p(\rho+\ep X)=\ep^2(p_{\rho}(X)+O(\ep)),\quad X=(x_0,x,\xi_0,\xi),\;\; (\ep\to 0) 
\]
can be reduced, up to a symplectic change of coordinates, to a normal form. Depending on whether $
{\rm Ker\,}F_p^2(\rho)\cap {\rm Im\,}F_p^2(\rho)=\{0\}$ or ${\rm Ker\,}F_p^2(\rho)\cap {\rm Im\,}F_p^2(\rho)\neq\{0\}$ 
different normal forms arise, leading to the following quadratic hyperbolic operators:
\begin{align*}
P_a&=-D_0^2+\Sigma_{j=1}^k\mu_j(D_j^2+x_j^2D_n^2)+\Sigma_{j=k+1}^{k+\ell}D_j^2,\\
P_b&=-D_0^2+2x_1D_0D_n+D_1^2+\Sigma_{j=1}^k\mu_j(D_j^2+x_j^2D_n^2)+\Sigma_{j=k+1}^{k+\ell}D_j^2
\end{align*}
where $n> k+\ell$ and $\rho=(0, 0, e_n)$. It is well known that the Cauchy problem for $P_a+SD_n$ is not well-posed in the Gevrey class $s>2$ if $S\in \C\setminus [-\Sigma \mu_j, \Sigma \mu_j]$ (see, for example \cite{Ho:Q}, \cite[Lemma 6.7]{Ni:book}), while it is well known that the Cauchy problem for $P$ is well-posed in the Gevrey class $1<s<2$ for all lower-order terms, which is a special case of a general result proved in \cite{Bro} (see also, \cite{Iv}). Therefore when ${\rm Ker\,}F_p^2(\rho)\cap {\rm Im\,}F_p^2(\rho)=\{0\}$, the Gevrey class $2$ is optimal.

It remains to study the case
\begin{equation}
\label{eq:spaceW}
{\rm Ker\,}F_p^2(\rho)\cap {\rm Im\,}F_p^2(\rho)\neq \{0\}.
\end{equation}
 In what follows, we assume that the set $\Sigma$ of critical points of $p$ is a $C^{\infty}$ manifold on which $p$ vanishes exactly to order $2$ and the rank of $\sum_{j=0}^nd\xi_j\wedge dx_j$ is constant. Under the condition \eqref{eq:spaceW}, the information on the spectrum of $F_p$ alone is not sufficient to fully determine the behavior of bicharacteristics near $\Sigma$, which is, however, essential whether the problem is  $C^{\infty}$ well posed (see \cite[Chapter 6]{Ni:book}). In this paper, we prove 
\begin{thm}
\label{thm:g3}
Assume that \eqref{eq:spaceW} holds on $\Sigma$. 
The Cauchy problem for $P$ is well-posed in the Gevrey 
class $1< s< 3$ for all lower-order 
terms, and the Gevrey class $3$ is optimal.
\end{thm}
Here, the Gevrey class $3$ is optimal in the following sense. Consider a quadratic hyperbolic operator 
\begin{gather*}
P_{c}=P_{b}+x_1^3D_n^2
\end{gather*}
which verifies \eqref{eq:spaceW} because the term $x_1^3D_n^2$ does not affect the Hamilton map, while there is a bicharacteristic tangent to $\Sigma$. The Cauchy problem for $P_{c}+SD_n$ with $S\in \C\setminus [-\Sigma \mu_j, \Sigma \mu_j]$ is not locally solvable in any Gevrey class $s>3$ (\cite{BN:g3}, \cite{Ni:arxiv}). 
\begin{thm}
\label{thm:g4} 
Assume \eqref{eq:spaceW} and that there is no bicharacteristic tangent to $\Sigma$, then the Cauchy problem for $P$ is well-posed in the Gevrey 
 class $1< s< 4$ for all lower-order 
 terms, and the Gevrey class $4$ is optimal.
\end{thm}
To see the optimality of the Gevrey class $4$, consider the model quadratic hyperbolic operator $P_b$ 
which verifies the required assumptions. The Cauchy problem for $P_b+SD_n$ with $S\in \C\setminus [-\Sigma \mu_j, \Sigma \mu_j]$ is not locally solvable in any Gevrey class $s>4$ (\cite{Ho:Q}, \cite[Chapter 5]{Ni:book}). In particular, in \cite{Ho:Q} an explicit formula of the forward fundamental solution of $P_b+SD_n$ is obtained for every $S\in \C$ which is a distribution on the Gevrey class $4$.

We proved this result in our previous papers \cite{BN:g3} and \cite{BN:g4} under the restriction that the codimension of $\Sigma$ is $3$. In the present paper, we remove this restriction and thereby complete the result. A key new ingredient is the normal form of the principal symbol obtained in \cite{BN:arxiv}. This normal form makes it possible to extend the pseudodifferential weight introduced in \cite{BN:g3}, \cite{BN:g4} to a form applicable in the general setting. 

Since our analysis requires pseudodifferential operators with symbol of type $\exp{S^{\kappa}_{\rho,\delta}}$ acting on Gevrey spaces, we establish a composition formula for such operators. Although this formula is less precise than the one developed in \cite{NTa}, which was used in our earlier work, it is considerably easier to apply. The details are presented in the Appendix.


\section{Preliminaries}
\label{sec:prel}

We prove Theorems \ref{thm:g3} and \ref{thm:g4} showing the existence of a parametrix with finite propagation speed of micro supports for $\op{e^{\varphi}}\op{\hat P}\op{e^{-\varphi}}$ as in \cite[Chapter 7]{Ni:book}, where $\varphi$ are suitably chosen $S_{\rho,\delta}$ type symbols and $\hat P$ is a localized symbol  of $P$ at $\rho\in \Sigma$.

By a suitable change of coordinates $(t, x)$, leaving the $t$-coordinate invariant, we can assume that $p$ has the form $-\tau^2+a_2(t,x,\xi)$.  Let's fix any $\bar\rho \in \Sigma$, and we work in a conic neighborhood of $\bar\rho$. Let the codimension of $\Sigma$ be $d+1$, then in a neighborhood of $\bar\rho$ we can write 
\[
p=-\tau^2+\Sigma_{j=1}^d\phi_j^2(t,x,\xi)
\]
where $d\phi_j$ (we often write $\tau$ as $\phi_0$) are linearly independent at $\rho$ and $\Sigma$ is given by $\phi_j=0$, $0\leq j\leq d$. 
Denote 
\[
\Sigma'=\{\phi_j(t,x,\xi)=0, 1\leq j\leq d\}
\]
and by $O^k(\Sigma')$ a smooth $f(t,x,\xi)$ vanishing of order $k$ on $\Sigma'$ near $\bro$ and write $O(\Sigma')$ for $O^1(\Sigma')$. We shall write $f_1\sieq f_2$ to mean $f_1-f_2=O(\Sigma')$.
Thanks to \cite[Proposition 2.1]{BN:arxiv}, under the condition \eqref{eq:spaceW}  one can write
\begin{equation}
\label{eq:form:0}
p=-(\tau+\phi_1)(\tau-\phi_1)+\Sigma_{j=2}^r\phi_j^2+\Sigma_{j=r+1}^d\phi_j^2
\end{equation}
in a conic neighborhood of $\bro$ where
\begin{align}
\label{eq:form:1}
\{\phi_i,\phi_j\}\sieq 0,\quad 0\leq i\leq d,\;\;j\geq r+1,\;\; \{\tau-\phi_1,\phi_j\}\sieq 0,\;\; 0\leq j\leq d,\\
\label{eq:form:2}
\{\phi_1,\phi_2\}(\bro)\neq 0, \{\phi_2,\phi_j\}\sieq 0, 2\leq j\leq r, {\rm det}(\{\phi_i,\phi_j\}(\bro))_{3\leq i,j\leq r}\neq 0.
\end{align}
As for the Gevrey $4$ case, note that the non-existence of bicharacteristics tangent to $\Sigma$ is equivalent to (see \cite{Ni:KJM}, also \cite{BN:arxiv}) 
\begin{equation}
\label{eq:cond:G4}
\{\{\tau-\phi_1,\phi_2\},\phi_2\}\sieq 0.
\end{equation}
Note that \eqref{eq:form:1} and \eqref{eq:form:2} imply $\{\tau, \phi_2\}(\bro)\neq 0$ so that $\dif_t\phi_2(\bro)\neq 0$, then one can write
\[
\phi_2=e_2(t-\kappa_2(x,\xi)),\quad e_2(\bro)\neq 0.
\]
Without restrictions we can assume $\bro=(0,e_n)$ with $e_n=(0,\ldots,0,1)$. In view of \eqref{eq:form:2} it follows that $d\kappa_2(\bro)\neq 0$ because $\phi_1$ is independent of $\tau$. Therefore take $\Xi_0=\xi_0$, $X_0=x_0$ and $X_1=\kappa_2(x,\xi)$ which satisfy the commutation relations with linearly independent $d\Xi_0$, $dX_0$, $dX_1$, $\Sigma_{j=0}^n\xi_jdx_j$ at $\bro$, hence extends to a full homogeneous symplectic coordinates system $(X,\Xi)$ (\cite[Theorem 21.1.9]{Ho:book3}). Thus, one can assume that
\[
\phi_2=e_2(t-x_1),\quad e_2(\bro)>0.
\]
Since \eqref{eq:form:2} implies that $\dif_{\xi_1}\phi_1(\bro)\neq 0$ one can write $
\phi_1=e_1(\xi_1-\kappa_1(t,x,\xi'))$ where $e_1(\bro)\neq 0$ and $\xi'=(\xi_2,\ldots,\xi_n)$. 
Write $\kappa_1(t,x,\xi')=\kappa_1(x_1,x,\xi')+c_1\phi_2$. Set $X_1=x_1$, $\Xi_1=\xi_1-\kappa_1(x_1,x,\xi')$ which satisfy the commutation relations with linearly independent $dX_1$, $d\Xi_1$, $\Sigma_{j=1}^d\xi_jdx_j$, hence we can again assume $\phi_1=e_1(\xi_1+c_1\phi_2)$. Since $\{\tau-\phi_1,t-x_1\}=O(\Sigma')$ we see that
 $e_1=-1+O(\Sigma')$ and hence
 \begin{equation}
 \label{eq:phi:1:a}
\phi_1=-\xi_1+c'_1\phi_2+\varPhi_1
\end{equation}
where $\varPhi_1$  is a quadratic function in $(\xi_1,\phi_2,\ldots,\phi_d)$.  Let $\chi(s)=s$ for small $|s|$ and nonzero constant outside a small neighborhood of $s=0$ and denote
\[
\psi_2=t-\chi(x_1),\quad \phi_2=e_2\psi_2\xiga,\quad \xiga^2=\ga^2+|\xi|^2
\]
where $e_2$ is extended to be a positive constant outside a small conic neighborhood of $\bro$.  We extend $\phi_j$ ($3\leq j\leq d$) to be $0$ outside a small conic neighborhood of $\bro$, homogeneous of degree $1$, and set $\psi_j(t,x,\xi)=\phi_j(t,x,\xi)\xiga^{-1}$, $3\leq j\leq d$. We denote
\begin{equation}
\label{eq:psi:1}
\psi_1=-\xi_1\xiga^{-1}+c'_1e_2\psi_2,\quad \phi_1=\psi_1\xiga+r_1
\end{equation}
where $r_1$ is obtained from $\varPhi_1$ replacing $(\phi_2,\ldots,\phi_d)$ by extended ones. Next, let $\chi_0(x,\xi)\geq 0$, homogeneous of degree $0$, be such that $\chi_0=0$ in a small conic neighborhood of $\bro$ and $1$ outside another small neighborhood and define $\psi_{d+1}$ and $\phi_{d+1}$ by
\begin{equation}
\label{eq:d+1:teigi}
\psi_{d+1}(t,x,\xi')=\chi_0(t-x_1,x',0,\xi'),\quad \phi_{d+1}=\psi_{d+1}\xiga
\end{equation}
so that $\{\tau+\xi_1,\psi_{d+1}\}=0$ and $\{t-x_1,\psi_{d+1}\}=0$. Since $|\xi/\xiga-\bro|<\ep$ if $|\xi/|\xi|-\bro|<\ep/2$ and $|\xi|\geq \ep^{-1/2}\ga$,  
we see that
\[
\psi_j\in S^{(s)}_{1,0}(1),\quad  1\leq j\leq d+1
\]
where $S_{\rho,\delta}^{(s)}$ denotes the symbol classes of type $S_{\rho,\delta}$ with Gevrey regularity, defined in Definition \ref{dfn:gmarus}. 
\begin{lem}
\label{lem:pre:a} One can write
\begin{gather*}
\{\tau-\phi_1,\psi_j\}=\Sigma_{l=1}^{d+1}c_{jl}\psi_{l},\quad 1\leq j\leq d+1,\quad |t|\leq t_0,\\
\{\psi_2,\phi_j\}=\Sigma_{l=1}^{d+1}c'_{jl}\psi_{l},\quad 2\leq j\leq d+1,\quad |t|\leq t_0.
\end{gather*}
\end{lem}
\begin{proof}From \eqref{eq:d+1:teigi} and the second condition in \eqref{eq:form:1} one can write $\{\tau-\phi_1,\phi_j\}=\Sigma_{\mu=1}^{d}c_{j\mu}\phi_{\mu}$ for $1\leq j\leq d+1$ in a conic neighborhood of $\bro$. Outside this conic neighborhood, one can assume $\chi(\xi_1/\ep'\xiga)(-\psi_1+c'_1e_2\psi_2)+\psi_{d+1}\geq c>0$ for \eqref{eq:psi:1}, choosing small $\ep'>0$.  Hence, the proof is immediate. The proof of the second assertion is similar.
\end{proof}
In the Gevrey $4$ case, with $\phi'=(\phi_1,\ldots,\phi_d)$ and $\phi''=(\phi_2,\ldots,\phi_d)$, from \eqref{eq:form:2}, \eqref{eq:cond:G4} and \eqref{eq:phi:1:a} one can write
%
\begin{equation}
\label{eq:phi:1:katati}
\phi_1=-\xi_1+c_1\phi_2+(\Sigma_{j=2}^dc_{1j}\phi_j)\xi_1+O(|\phi''|^2)+O(|\phi'|^3).
\end{equation}
%
Denoting $\bar\phi_1=\phi_1-\psi$ with $\psi=(\Sigma_{j=3}^dc_{1j}\phi_j)\xi_1$, we see that $\{\tau-\bar\phi_1,\phi_2\}=c'\phi_2+O(|\phi'|^2)$ near $\bro$, which follows from \eqref{eq:form:2}. Thus, in the Gevrey $4$ case, a part of the assertions in Lemma \ref{lem:pre:a} can be improved such that
\begin{equation}
\label{eq:G4:bunkai}
\begin{split}
\{\tau-\bar\phi_1, \psi_2\}=c_2\psi_2+\Sigma_{i,j=1}^{d+1}c_{ij}\psi_i\psi_j,\\
\{\psi_2,\bar\phi_1\}=\{\xi_1,\psi_2\}+c_2'\psi_2+\Sigma_{i,j=1}^{d+1}c'_{ij}\psi_i\psi_j.
\end{split}
\end{equation}
In what follows we denote $\bar\phi_1$ by $\phi_1$ so that
\begin{equation}
\label{eq:G4:kakikae}
p=-(\tau-\phi_1)(\tau+\phi_1)+\Sigma_{j=2}^d\phi_j^2-2\psi\phi_1+\psi^2.
\end{equation}

Adding $k_2\xiga^{\kappa}\psi^2_{d+1}(\tau-\phi_1)+k_1\phi_{d+1}^2$ to $p$ in \eqref{eq:form:0} or \eqref{eq:G4:kakikae}, we consider
\[
p+k_2\xiga^{\kappa}\psi_{d+1}^2(\tau-\phi_1)+k_1\phi_{d+1}^2
\]
as new $p$, where $k_1, k_2> 1$ are parameters, and for simplicity we write $\phi_{d+1}$ for $\sqrt{k_1}\phi_{d+1}$. Here, $\kappa>0$ will be specified later. 
%
%
%
\section{Weights $w$ and $\phi$}

We define the weight symbol $w(t, x,\xi)$ by
\[
w=\big(\varep \Sigma_{j=1}^{d+1}\psi_j^{6}(t, x,\xi)+\ell^2\xiga^{-2}\big)^{1/4}\;\;\text{or}\;\; w=\big(\Sigma_{j=1}^{d+1}\psi^4_j(t,x,\xi)
+\ell\xiga^{-1}\big)^{1/2}
\]
according to whether we are in the Gevrey $3$ or $4$ case. Here $\ell>0$ is a large parameter subject to the constraint $2\ell \leq \ga$ and $\varep>0$ is chosen so that  $\varep\Sigma_{j=1}^{d+1}\psi_j^{6}(t, x,\xi)\leq 1/4$. Note that
\begin{equation}
\label{eq:doto:1}
\Sigma_{j=1}^{d+1}\psi_j^2/C\leq \Sigma_{j=1}^{d+1}\phi_j^2\xiga^{-2}\leq C\Sigma_{j=1}^{d+1}\psi_j^2.
\end{equation}
In the following arguments, the specific value of $\varep$ is irrelevant, so we may assume $\varep=1$ without loss of generality. Denote
\[
r(t, x, \xi)=\sqrt{\psi_2^2(t,x)+w^2(t,x,\xi)},\quad \psi_2=t-\chi(x_1).
\]
Write $X=(x,\xi), Y=(y,\eta)\in\R^{2n}$ and introduce two metrics 
\begin{gather*}
{\munderbar g}_X(Y)=w(X)^{-4\delta}(|y|^2+\xiga^{-2}|\eta|^2),\\
{\bar g}_X(Y)=w(X)^{-2}|y|^2+w(X)^{-4\delta}\xiga^{-2}|\eta|^2
\end{gather*}
where $\delta=1/k$ according to the Gevrey $k$ case ($k=3$ or $4$), and $\rho$ is given by $\rho+\delta=1$. 
Since $w^{-1}\leq \xiga^{1/2}$ and $w\leq 1$ it is clear that 
\[
{\munderbar g}\leq g_{\rho, \delta}\leq g_{\rho,1/2},\quad {\munderbar g}\leq {\bar g}\leq g_{\rho,1/2}
\]
where $g_{\rho,\delta}=\xiga^{2\delta}|y|^2+\xiga^{-2\rho}|\eta|^2$ is the metric definig $S_{\rho,\delta}$ class.
\begin{lem}
\label{lem:ex:1}There exist $C, A>0$ such that 
\begin{equation}
\label{eq:w:p:1}
|\dif_{x}^{\be}\dif_{\xi}^{\al}w|\leq CA^{|\al+\be|}|\al+\be|!^{s}ww^{-2\delta|\al+\be|}\xiga^{-|\al|}
\end{equation}
thus $w\in S^{(s)}_{\rho, \delta}(w)$. We have also $w\in S(w, {\munderbar g})$, and hence $w^t\in S(w^t, {\munderbar g})$ $(t\in\R)$, moreover $w$ is $S_{\rho,\delta}$ admissible weight, defined in Definition \ref{dfn:admissible}.
\end{lem}
\begin{proof}We only show that $w$ is $S_{\rho,\delta}$ admissible. Denote $g=g_{\rho,\delta}$. Thanks to \eqref{eq:w:p:1} we have $|\dif_x^{\be}\dif_{\xi}^{\al}w^{2\delta}|\leq C\xiga^{-|\al|}$ for $|\al+\be|=1$. Then
\[
|w^{2\delta}(X+Y)-w^{2\delta}(X)|\leq C(|y|+\lr{\xi+s \eta}_{\ga}^{-1}|\eta|),\quad |s|<1.
\]
If $|\eta|\leq \xiga/2$ so that $ \lr{\xi+s \eta}_{\ga}\approx \xiga$ the right-hand side is bounded by $C(|y|+\xiga^{-1}|\eta|)\leq C\xiga^{-\delta}g_X^{1/2}(Y)\leq Cw^{2\delta}(X)g_X^{1/2}(Y)$. If $|\eta|\geq \xiga/2$ then $g_X(Y)\geq \xiga^{2\delta}/4$. Therefore $w^{2\delta}(X+Y)\leq C\leq C'\xiga^{-\delta}g_X^{1/2}(Y)\leq C'w^{2\delta}(X)g_X^{1/2}(Y)$ shows that $w^{2\delta}$ is $g$ admissible weight and so is $w$.
\end{proof}
\begin{lem}
\label{lem:ex:2}We have
\[
|\dif_x^{\be}\dif_{\xi}^{\al}r|\leq CA^{|\al+\be|}|\al+\be|!^srw^{-|\be|-2\delta|\al|}\xiga^{-|\al|}
\]
hence $r\in S^{(s)}_{\rho, 1/2}(r)$ and $r\in S(r, {\bar g})$.  Moreover $r$ is $S_{\rho,1/2}$ admissible.
\end{lem}
\begin{proof}
It suffices to show that $r^2=\psi_2^2(x)+w^2$ is $S_{\rho, 1/2}$ admissible weight. Since $w^2$ is $S_{\rho,\delta}$ admissible by Lemma \ref{lem:ex:1}, hence  $S_{\rho,1/2}$ admissible. Denoting $g=g_{\rho,1/2}$ we see that $|\psi_2(X+Y)-\psi_2(X)|\leq C|y|\leq C\xiga^{-1/2}g_X^{1/2}(Y)\leq Cw(X)g_X^{1/2}(Y)$ proving
$\psi_2^2(X+Y)\leq C(\psi_2^2(X)+w^2(X))(1+g_X(Y))$, 
hence we conclude the assertion.
\end{proof}
It is easily checked that $\phi_1\in S(w^{2\delta}\xiga, \munderbar g)$. Now introduce the symbol;
\begin{gather*}
\phi(x,\xi)=i\big\{\log{(\psi_2(x)-iw(x,\xi))}-\log{(\psi_2(x)+iw(x,\xi))}\big\}\\=2\arg{(\psi_2+iw)}.
\end{gather*}
\begin{lem}
\label{lem:ex:3}
We have $\phi\in S^{(s)}_{\rho, 1/2}(\phi)$, $\phi\in S(\phi, {\bar g})$ and for $|\al+\be|=1$ we have $\dif_x^{\be}\dif_{\xi}^{\al}\phi\in S^{(s)}(wr^{-1}w^{-|\be|-2\delta|\al|}\xiga^{-|\al|}, {\bar g})$. Moreover $\phi$ is $S_{\rho,1/2}$ admissible.
\end{lem}
\begin{proof}For $|\al+\be|=1$ one has
\begin{equation}
\label{eq:phi:1:bibun}
\dif_x^{\be}\dif_{\xi}^{\al}\phi=-2r^{-2}(x,\xi)\big(w(x,\xi)\dif_x^{\be}\dif_{\xi}^{\al}\psi_2(x)-\psi_2(x)\dif_x^{\be}\dif_{\xi}^{\al}w(x,\xi)\big)
\end{equation}
where $\psi_2(x)\dif_x^{\be}\dif_{\xi}^{\al}w\in S^{(s)}(rw^{1-2\delta|\al+\be|}\xiga^{-|\al|}, {\munderbar g})$ in view of Lemma \ref{lem:ex:1}, thus the second assertion.
Since $|\psi_2(x)|\leq C$ there is $c>0$ such that
\begin{equation}
\label{eq:phi:sita}
\phi=2\arg{(\psi_2+iw)}=2\arctan{(w/r)}\geq c\,w/r
\end{equation}
then thanks to Lemmas \ref{lem:ex:1} and \ref{lem:ex:2}, it follows that 
\begin{gather*}
w/r^2\in S(w/r^2, {\bar g})\subset S(w^{-1}\phi, {\bar g}),\\
\psi_2\dif_x^{\be}\dif_{\xi}^{\al}w/r^2\in S(w^{1-2\delta|\al+\be|}\xiga^{-|\al|}/r, {\bar g})\subset S^{(s)}(w^{-2\delta|\al+\be|}\xiga^{-|\al|}\phi, {\bar g})
\end{gather*}
which together with \eqref{eq:phi:1:bibun} shows $\phi\in S(\phi, {\bar g})$. Next, in view of \eqref{eq:phi:1:bibun} we have
\begin{gather*}
|\phi(X+Y)-\phi(X)|\leq C\big(w/r^2+w^{1-2\delta}/r\big)\big|_{(X+s_1 Y)}|y|\\
+C\big(w^{1-2\delta}/r\big)\big|_{(X+s_1 Y)}\lr{\xi+s_1 \eta}_{\ga}^{-1}|\eta|,\quad |s_1|<1.
\end{gather*}
If $|\eta|\leq \xiga/2$ so that $g_X\approx g_{X+s_1 Y}$ with $g=g_{\rho,1/2}$, recalling that $w$ and $r$ are $S_{\!\rho,1/2}$ admissible we have
\begin{gather*}
w(X+s_1 Y)/r^2(X+s_1 Y)\leq C(w(X)/r^2(X))(1+g_X(Y))^N,\\
w^{1-2\delta}(X+s_1 Y)/r(X+s_1 Y)\leq C(w^{1-2\delta}(X)/r(X))(1+g_X(Y))^N
\end{gather*}
from which it follows that
\begin{gather*}
|\phi(X+Y)-\phi(X)|\leq C\phi(X)(\xiga^{1/2}|y|+\xiga^{-\rho}|\eta|)(1+g_X(Y))^N\\
\leq C'\phi(X)(1+g_X(Y))^{N+1/2}
\end{gather*}
since $r(X)\geq w(X)\geq \xiga^{-1/2}$. If $|\eta|\geq \xiga/2$ so that $g_X(Y)\geq \xiga^{2\delta}/4$ noting that $\phi(X)\geq c\xiga^{-1/2}$ in view of \eqref{eq:phi:sita}, we have
\begin{gather*}
\phi(X+Y)\leq 2\pi\leq C\xiga^{-1/2}(1+g_X(Y))^{1/4\delta}\leq C\phi(X)
(1+g_X(Y))^{1/4\delta}
\end{gather*}
completing the proof.
\end{proof}
We need more precise estimates on derivatives of $\phi$. Denote
\[
\N^{2n}\ni \al=(\al_{x_1},\ldots,\al_{x_n},\al_{\xi_1},\ldots,\al_{\xi_n})=(\al_x, \al_{\xi}),\quad \al^{\sigma}=(\al_{\xi},\al_x).
\]
\begin{lem}
\label{lem:phi:precise}Writing $\al=\tilde\al+\hat\al$ with $|\hat\al|=1$ we have
\begin{gather*}
\dif_X^{\al}\phi\in S(w^{1-2\delta}r^{-1}w^{-2\delta|\tilde\al|-(1-2\delta)\tilde\al_{x_1}}\xiga^{\kappa-|\al_{\xi}|}, \bar g), \quad \dif_X^{\hat\al}\neq \dif_{x_1},\\
\dif_X^{\al}\phi\in S(wr^{-2}w^{-2\delta|\tilde\al|-(1-2\delta)\tilde\al_{x_1}}\xiga^{\kappa-|\al_{\xi}|}, \bar g), \quad \dif_X^{\hat\al}= \dif_{x_1}.
\end{gather*}
In particular, we have
\begin{gather*}
\dif_X^{\al}\phi\in S(w^{1-2\delta}r^{-1}w^{-2\delta|\al|-(1-2\delta)\al_{x_1}+2\delta}\xiga^{\kappa-|\al_{\xi}|}, \bar g),\quad |\al|\geq 1.
\end{gather*}
\end{lem}
\begin{proof}
Recall $\dif_X^{\al}\phi=-2r^{-2}(x,\xi)\big(w(x,\xi)\dif_X^{\al}\psi_2(x)-\psi_2(x)\dif_X^{\al}w(x,\xi)\big)$ for $|\al|=1$.
Since $\dif_X^{\al}\psi_2=0$ unless $\dif_X^{\al}=\dif_{x_1}$ 
so that
\begin{gather*}
\dif_X^{\al}\phi\in S(w^{1-2\delta}r^{-1}\xiga^{-|\al_{\xi}|}, \bar g), \quad \dif_X^{\al}\neq \dif_{x_1},\quad |\al|=1,\\
\dif_X^{\al}\phi\in S(wr^{-2}, \bar g), \quad \dif_X^{\al}= \dif_{x_1}
\end{gather*}
from which we have the first assertion. The second assertion is clear.
\end{proof}

\section{Composition by PDO with symbol of type $\exp{S_{\rho,\delta}}$}
\subsection{Preliminary composition}

Recall that $w\in S^{(s)}_{\rho,\delta}(w)$ and $S_{\rho,\delta}$ admissible by Lemma \ref{lem:ex:1} and $r\in S^{(s)}_{\rho,1/2}(r)$ is $S_{\rho,1/2}$ admissible by Lemma \ref{lem:ex:2}. Let  $\kappa'>0, \kappa>0$ be such that
\begin{equation}
\label{eq:k:k}
\kappa'=\delta+\ep,\quad \kappa=1/2-\delta-\ep\quad (\kappa'+\kappa=1/2)
\end{equation}
where $\ep>0$ is enough small, and we recall that $\delta=1/k$ if we are in the Gevrey $k$ case ($k=3, 4$). Consider 
\[
e^{-\te\lr{D}_{\!\ga}^{\kappa'}(t-\tka
)}Pe^{\te\lr{D}_{\!\ga}^{\kappa'}(t-\tka
)
},\quad T(\xi)=e^{-\te\xiga^{\kappa'}(t-\tka)},\;\;\tilde T=e^{\te\xiga^{\kappa'}(t-\tka)}
\]
where $0\leq t\leq \tka$ with a small $\tka
>0$ and 
$\theta>0$ is a positive parameter. We often use $T$ ($\tilde T$) for both $T(\xi)$ ($\tilde T(\xi)$) and $T(D)$ ($\tilde T(D)$), but it is clear from the context. 
Let $f\in S^{(s)}_{1,0}(1)$ and consider
$(\sigma \dif_X)^{\al}(\dif_X^{\al^0}f \dif_X^{\al^1}\xiga^{\kappa'}\cdots\dif_X^{\al^k}\xiga^{\kappa'})$ with $\al=\al^0+\al^1+\cdots+\al^k$ which belongs to $S^{(s)}_{1,0}(\xiga^{-(1-\kappa')k-(|\al|-k)})$. Since $2(1-\kappa')\geq 1$ for small $\ep$, thanks to Theorem \ref{thm:matome} we have
\begin{gather*}
T\#f\#\tilde T=f+i\te (t-\tka)\{\xiga^{\kappa'},f\}
+S^{(s)}_{1,0}(\xiga^{-1})+S^{(s)}_{0,0}(e^{-c\xiga^{1/s}}).
\end{gather*}
Therefore, we have
\begin{equation}
\label{eq:T:p:T}
T\#(-\tau^2+\Sigma_{j=1}^{d+1}\phi_j^2)\#\tilde T=-(\tau-i\te\xiga^{\kappa'})^2+\Sigma_{j=1}^{d+1}\phi_j^2+\bar q_1+\bar q_2+q_R
\end{equation}
where $\bar q_1=2i\te (t-\tka)\{\xiga^{\kappa'}, \Sigma_{j=1}^{d+1}\phi_j^2\}$, $\bar q_2\in S^{(s)}_{1,0}(\xiga)$ and $q_R\in S^{(s)}_{0,0}(e^{-c\xiga^{1/s}})$. 
Writing $\{\xiga^{\kappa'},\phi_1^2\}=c\,\xiga^{\kappa'}\phi_1$ with $c\in S^{(s)}_{1,0}(1)$ and 
move this term inside the parentheses so that
\begin{gather*}
-(\tau-i\te\xiga^{\kappa'})^2+\phi_1^2+2i\theta (t-\tka) c\,\xiga^{\kappa'}\phi_1\\
=-(\tau-i\te c_1\xiga^{\kappa'}-\phi_1)(\tau-i\te c_2\xiga^{\kappa'}+\phi_1)-\te^2(t-\tka)^2\xiga^{2\kappa'}
\end{gather*}
where $S^{(s)}_{1,0}(1)\ni c_i=1-(-1)^ic\,(t-\tka)\geq 1/2$ for small $\tka>0$.
Denote
\[
\tilde\phi_1=\phi_1-\tilde w\phi_1,\quad \tilde w=1-\sqrt{1-w^{2/3}}=w^{2/3}/\big(1+\sqrt{1-w^{2/3}}\big)
\]
then we can rewrite the right-hand of \eqref{eq:T:p:T} as
\begin{equation}
\label{eq:T:p:T:bis}
\begin{split}
-(\tau-i\te c_1\xiga^{\kappa'}-\tilde\phi_1)(\tau-i\te c_2\xiga^{\kappa'}+\tilde\phi_1)\\
+2\tilde w\phi^2_1(1-\tilde w/2)+\Sigma_{j=2}^{d+1}\phi_j^2
+\tilde q_1+\tilde q_2+q_R
\end{split}
\end{equation}
with $\tilde q_1=2i\te (t-\tka)(c\tilde w\phi_1\xiga^{\kappa'}+\{\xiga^{\kappa'},\Sigma_{j=2}^{d+1}\phi_j^2\})$ and $\tilde q_2\in S^{(s)}_{1,0}(\xiga)$, where $\tilde w\phi_1
\in S^{(s)}_{\rho,\delta}(\xiga)\cap S(w^{4\delta}\xiga, \munderbar g)$. 

In the Gevrey $4$ case, we choose $\tilde w=\phi_1^2\xiga^{-2}\in S^{(s)}_{1, 0}(\xiga)\cap S(w^{4\delta}\xiga, \munderbar g)$ so that we have \eqref{eq:T:p:T:bis} up to the term $-2\psi\phi_1+\psi^2$. Here, we remark that one can assume
\[
2\tilde w\phi^2_1(1-\tilde w/2)+\Sigma_{j=2}^{d+1}\phi_j^2
-2\psi\phi_1+\psi^2
\geq c\,(|\phi''|^2+\phi_1^4\xiga^{-2})
\]
since $|(\phi_3,\ldots,\phi_d)|$ can be assumed to be arbitrarily small, choosing suitable their extensions. Denote
\[
\la_i=i\te c_i\xiga^{\kappa'}+\tilde\phi_1,\quad q=2\tilde w\phi^2_1(1-\tilde w/2)+\Sigma_{j=2}^{d+1}\phi_j^2+\ell\xiga
\]
where $-2\psi\phi_1+\psi^2$ should be added to $q$ in the Gevrey 4 case. An application of Theorem \ref{thm:matome:a:1:2} shows that the first line of \eqref{eq:T:p:T:bis} is
\begin{gather*}
-(\tau-i\te c_2\xiga^{\kappa'}-\la_2)\#(\tau-i\te c_1\xiga^{\kappa'}-\la_1)\\
-2\tilde\phi_1\#(\tau-\la_1-i\te c_1\xiga^{\kappa'})
+S^{(s)}_{\rho,\delta}(\xiga)+S^{(s/(1-\delta))}_{0,0}(e^{-c\xiga^{(1-\delta)/s}}).
\end{gather*}
As for $\xiga^{\kappa}\psi_{d+1}^2(\tau-\phi_1)$, we write this as $
(\xiga^{\kappa}\psi_{d+1}^2)\#(\tau-\phi_1)+S^{(s)}_{1,0}(\xiga^{\kappa})+S^{(s)}_{0,0}(e^{-c\xiga^{1/s}})$ by use of Theorem \ref{thm:matome:a:1:2}, 
 hence we have
 \begin{gather*}
 T\#(\xiga^{\kappa}\psi_{d+1}^2(\tau-\phi_1))\#\tilde T=\big(\xiga^{\kappa}\psi_{d+1}^2+S^{(s)}_{1,0}(1)+S_{0,0}^{(s)}(e^{-c\xiga^{1/s}})\big)\\
 \#\big(\tau-i\te\xiga^{\kappa'}-\phi_1+S_{1,0}^{(s)}(\xiga^{\kappa'})+S^{(s)}_{1,0}(1)+S_{0,0}^{(s)}(e^{-c\xiga^{1/s}})\big).
 \end{gather*}
 Since $S_{0,0}^{(s)}(e^{-c\xiga^{1/s}})$ can be managed easily by Proposition \ref{pro:b:class}, applying Theorem \ref{thm:matome:a:1:2}, it suffices to study  
 \[
 (\xiga^{\kappa}\psi_{d+1}^2+S^{(s)}_{1,0}(1))\#(\tau-i\te c_1\xiga^{\kappa'}-\tilde\phi_1)+\tilde w \phi_1\psi_{d+1}^2\xiga^{\kappa}.
 \]
 %
%
\begin{lem}
\label{lem:p:exp}One can write
\begin{gather*}
T\#p\#\tilde T=-(\tau-\la_2)\#(\tau-\la_1)
-(2\tilde\phi_1+ik_2\xiga^{\kappa}\psi_{d+1}^2+S^{(s)}_{\rho,\delta}(1))\#(\tau-\la_1)\\
+q+q'_1
+S^{(s)}_{\rho,\delta}(\xiga)+S^{(s/(1-\delta))}_{0,0}(e^{-c\xiga^{(1-\delta)/s}})
\end{gather*}
with $q_1'=2i\te (t-\tka) (c\tilde w\phi_1\xiga^{\kappa'}+2\{\xiga^{\kappa'},\Sigma_{j=2}^{d+1}\phi_j^2\})+ik_2\tilde w \phi_1\psi_{d+1}^2\xiga^{\kappa}$, here we need to keep in mind that $S^{(s)}_{\rho,\delta}(\xiga)$ contains the term $-\ell\xiga$.
\end{lem}
%

%
\subsection{Composition $e^{\varphi}\#\la\#e^{-\varphi}$}

Introduce the main pseudodifferential weight. Let $\chi(x,\xi)$ be $1$ in a small neighborhood of $\bro$ and $\chi=-1$ outside another small neighborhood. Define
\begin{gather*}
\varphi=-\xiga^{\kappa}\phi(t,x,\xi)-k_3\xiga^{\kappa}\chi(x,\xi)\in S^{(s)}_{\rho, \delta}(\xiga^{\kappa})
\end{gather*}
with a parameter $k_3\gg 1$. We consider $e^{\varphi}\#T\#p\#\tilde T\#e^{-\varphi}$, but since $\xiga^{\kappa}\chi(x,\xi)$ belongs to $S^{(s)}_{1,0}(\xiga^{\kappa})$ and the treatment of $e^{\pm k_3\xiga^{\kappa}\chi}$ is the same as that of $T$, we restrict ourselves to the case $\varphi=-\xiga^{\kappa}\phi(t,x,\xi)$ from now on. Note that $e^{\pm\varphi}\in S_{\rho,1/2}^{(s)}(e^{c\xiga^{\kappa}})\subset {\mathcal A}_{1/2}^{(s/(1-\delta))}(e^{c\xiga^{\kappa}})$ by Corollary \ref{cor:bb}, and that one can take $\kappa<\tilde\kappa<(1-\delta)/s$ such that $\tilde \kappa s/(1-\delta)<1/2$ in view of \eqref{eq:k:k}. Then it follows from Proposition \ref{pro:b:class} that 
\begin{equation}
\label{eq:amari}
\begin{split}
e^{\varphi}\#S^{(s/(1-\delta))}_{0,0}(e^{-c\xiga^{(1-\delta)/s}})\#e^{-\varphi}\\\in {\mathcal A}^{(s/(1-\delta))}_{1/2}(e^{-c\xiga^{\tilde\kappa}})
\subset S_{\rho,1/2}(\xiga^{-N})
\end{split}
\end{equation}
for any $N\in\N$. It follows from Lemma \ref{lem:phi:precise} that 
\[
\dif_X^{\al^1}\varphi\cdots\dif_X^{\al^k}\varphi\in S(w^{1-2\delta}r^{-1}w^{-2\delta|\tilde\al|-(1-2\delta)\tilde\al_{x_1}+2\delta}\xiga^{k\kappa-|\tilde\al_{\xi}|}, \bar g)
\]
with $\tilde\al=\al^1+\cdots+\al^k$. Consider 
\begin{equation}
\label{eq:kigo}
\begin{split}
&J_{\al,k}(b)=\dif_X^{\hat\be}\dif_X^{\hat\al}b\,\dif_X^{\tilde\be}(\dif_X^{\al^1}\varphi\cdots\dif_X^{\al^k}\varphi),\quad \be=\al^{\sigma},\\
&\al=\hat\al+\tilde\al,\;\;\be=\hat\be+\tilde\be,\;\; \tilde\al=\al^1+\cdots+\al^k,\;\;\tilde\be=\be^1+\cdots+\be^k.
\end{split}
\end{equation}
Assume $\dif_X^{\be}b\in S(\bar w\xiga^{p-|\be_{\xi}|}, \bar g)$ for all $\be$. 
Since $\tilde\al_{x_1}=\be_{\xi_1}-\hat\al_{x_1}$ we see 
\[
\tilde\al_{x_1}+\tilde\be_{x_1}=\be_{\xi_1}+\tilde\be_{x_1}-\hat\al_{x_1}\leq |\be|-\hat\be_{x_1}-\hat\al_{x_1}\leq |\be|=|\al|
\]
then noting that $|\al|=|\hat\al+\hat\be+\tilde\al+\tilde\be|/2$ and $w^{-1}\leq \xiga^{1/2}$ we have 
\begin{equation}
\label{eq:q:furoku}
\begin{split}
J_{\al,k}(b)\in S(\bar w w^{1/2}r^{-1}w^{-2\delta|\tilde\al+\tilde\be|-(1-2\delta)(\tilde\al_{x_1}+\tilde\be_{x_1})+1/2}\xiga^{p+k\kappa-|\al|}, \bar g)\\
\subset S(\bar w \xiga^{\delta|\tilde\al+\tilde\be|+(1/2-\delta)(|\al|-\hat\al_{x_1}-\hat\be_{x_1})}\xiga^{p+k\kappa-|\al|}, \bar g)\\
\subset S(\bar w \xiga^{p-\delta|\hat\al+\hat\be|-(1/2-\delta)(|\hat\al|+k'+\hat\al_{x_1}+\hat\be_{x_1})-\ep k}, \bar g)\\
\subset S(\bar w \xiga^{p-\delta|\hat\al+\hat\be|-\ep k}, \bar g)
\end{split}
\end{equation}
where $|\al|-k=|\hat\al|+k'$ with $|\tilde \al|=k+k'$.  Hence we obtain 
\begin{equation}
\label{eq:q:furoku:b}
J_{\al,k}(b)\in S(\bar w \xiga^{p-\delta |\hat\al+\hat\be|-\ep k},\bar g).
\end{equation}
\begin{lem}
\label{lem:yobi}Let $b\in S_{1,0}(\xiga^{p})$. In the Gevrey $4$ case, we further  assume $\dif_{\xi_1}^2b\in S(w^{2\delta}\xiga^{p-2}, \bar g)$. Then $J_{\al,k}(b)\in S(\xiga^{p-1+\kappa'-\ep},\bar g)$ for $|\hat\al+\hat\be|\geq 2$. 
\end{lem}
\begin{proof}From \eqref{eq:q:furoku} one has $J_{\al,k}\in S(\xiga^{p-1/2-\delta-k\ep},\bar g)\subset S(\xiga^{p-1+\kappa'-\ep},\bar g)$ if $|\hat\al|+k'+\hat\al_{x_1}+\hat\be_{x_1}\neq 0$ for $|\hat\al+\hat\be|\geq 2$.   
If $|\hat\al|+k'+\hat\al_{x_1}+\hat\be_{x_1}=0$
so that $|\hat\al|=0$ and $\hat\be_{x_1}=0$ hence $\tilde\be_{x_1}=\al_{\xi_1}-\hat\be_{x_1}=\al_{\xi_1}$ one has $J_{\al,k}(b)\in S(\xiga^{p-2\delta-(1/2-\delta)(|\al|-\al_{x_1}-\al_{\xi_1})-\ep k}, \bar g)$, thus in  $S(\xiga^{p-1/2-\delta-k\ep},\bar g)$ if $|\al|-\al_{x_1}-\al_{\xi_1}\neq 0$. Assume $|\al|-\al_{x_1}-\al_{\xi_1}=0$, then denoting $\al_{\xi_1}=\mu$ and $\al_{x_1}=\nu$  we see $\dif_X^{\al}=\dif_{x_1}^{\nu}\dif_{\xi_1}^{\mu}$ and $\mu+\nu=|\al|$. Recalling $(\sigma\dif_X)^{\al}=(-1)^{\mu}\dif_{\xi_1}^{\nu}\dif_{x_1}^{\mu}$ and $\hat\be_{x_1}=0$, writing $\nu=\nu_1+\nu_2$,  $J_{\al,k}(b)$ is a sum of terms such as $\dif_{\xi_1}^{\nu_1}b\dif_{\xi_1}^{\nu_2}\dif_{x_1}^{\mu}((\dif_{\xi_1}\varphi)^{\mu}(\dif_{x_1}\varphi)^{\nu})$. In the Gevrey $4$ case, from Lemma \ref{lem:phi:precise} and the hypothesis this belongs to 
\[
S((wr^{-2})^{\nu}(w^{1-2\delta}r^{-1})^{\mu}w^{2\delta-2\delta \nu_2-\mu}\xiga^{p+k\kappa-|\al|}, \bar g)
\]
 for $\nu_1\geq 2$ which is contained in $S(\xiga^{p-3\delta -k\ep}, \bar g)\subset S(\xiga^{p-1+\kappa'-\ep},\bar g)$. In the Gevrey $3$ case, the term belongs to $S(\xiga^{p-2\delta -k\ep}, \bar g)\subset S(\xiga^{p-1+\kappa'-\ep},\bar g)$.
\end{proof}
\begin{lem}
\label{lem:b:key:key}
Let $b\in S_{1,0}^{(s)}(\xiga)$ with $\{b,\varphi\}\in S(\xiga^{1/2+\kappa}, \bar g)$. In the Gevrey $4$ case, we further  assume $\dif_{\xi_1}^2b\in S(w^{2\delta}\xiga^{-1}, \bar g)$. Then $e^{\varphi}\#b\#e^{-\varphi}$ is
\begin{gather*}
b(1-\mu)-((\{b,\varphi\}/2i)e^{\varphi})\#e^{-\varphi}
-e^{\varphi}\#((\{b,\varphi\}/2i)e^{-\varphi})\\
+S_{\rho,1/2}(\xiga^{\kappa'-\ep})+S_{\rho,1/2}(1),\quad e^{\varphi}\#e^{-\varphi}=1-\mu.
\end{gather*}
\end{lem}
\begin{proof}Since the sum over $|\hat\al+\hat\be|=0$ or $k=0$ gives $b(1-\mu)+S_{\rho,1/2}(\xiga^{-N}))$ 
since $\Sigma_{|\al|=r}(\sigma D_X)^{\al}\dif_X^{\al}f/\al!=0$ for $r\geq 1$, it is enough to study the sum over $|\hat\al+\hat\be|\geq 1$. From Theorem \ref{thm:matome}, it suffices to consider the case where $|\al|-k=|\hat\al|+k'$ is even. If $|\al|-k\geq 2$ we have $J_{\al,k}(b)\in S(\xiga^{\delta-k\ep},\bar g)\subset S(\xiga^{\kappa'-\ep},\bar g)$ by \eqref{eq:q:furoku}. If $|\hat\be|\geq 2$ we have $
J_{\al,k}(b)\in S(\xiga^{\kappa'-\ep}, \bar g)$ from Lemma \ref{lem:yobi}. Therefore we may assume $|\al|=k$ and $|\hat\be|=1$, that is we must study
\begin{equation}
\label{eq:hidari:comp}
\begin{split}
&\Big(\sum_{k=0}^m\frac{(-1)^k}{k!}{\tilde\sum_k}\frac{1}{(2i)^{|\al|}}+\sum_{k=0}^m\frac{1}{k!}{\tilde \sum_k}\frac{(-1)^{|\al|}}{(2i)^{|\al|}}\Big)\\
\times&\sum_{j=1}^k(\sigma \dif_Y)^{\al^j}b(X+2Y)
(\sigma\dif_Y)^{\tilde\al^j}(\dif_X\phi(X+Y))^{\al}
\big|_{Y=0}
\end{split}
\end{equation}
where $\tilde\al^j=\al^1+\cdots+\al^{j-1}+\al^{j+1}+\cdots+\al^k$ and the sum $\tilde\sum_k$ is  taken over all $\al^1+\cdots+\al^k=\al$, $|\al^j|=1$. Note that
\begin{gather*}
\frac{1}{k!}{ \tilde\sum_k}\frac{1}{(2i)^{|\al|}}\sum_{j=1}^k(\sigma \dif_Y)^{\al^j}b(X+2Y)
(\sigma\dif_Y)^{\tilde\al^j}(\dif_X\phi(X+Y))^{\al}|_{Y=0}\\
=\frac{1}{i(k-1)!}{\tilde \sum_{k-1}}\frac{1}{(2i)^{|\al|}}(\sigma\dif_X)^{\al}(\{b,\varphi\}(\dif_X\varphi)^{\al})+r_{\al}
\end{gather*}
where $r_{\tilde\al}$ consists of a sum of terms that contains $\dif_X^{\ga}b$ with $|\ga|\geq 2$ as a product factor, hence giving $S(\xiga^{\kappa'-\ep},\bar g)$ by Lemma \ref{lem:yobi}. Thanks to Proposition \ref{thm:matome:b}, \eqref{eq:hidari:comp} coincides with $-((\{b,\varphi\}/i)e^{\varphi})\#e^{-\varphi}-e^{\varphi}\#((\{b,\varphi\}/i)e^{-\varphi}$ modulo $S(\xiga^{\kappa'-\ep}, \bar g)$ because $J_{\al,  k}(\{b,\varphi\})\in S(\xiga^{-1/2+2\delta+\kappa-k\ep},\bar g)\subset S(\xiga^{\kappa'-\ep},\bar g)$ for $|\al|\geq k+2$ by \eqref{eq:q:furoku} since $\{b,\varphi\}\in S(\xiga^{1/2+\kappa},\bar g)$.
\end{proof}
\begin{lem}
\label{lem:b:key}
If $b\in S_{\rho,\delta}^{(s)}(\xiga^p)\cap S(w^{2l\delta}\xiga^p,\munderbar g)$ then for any $N\in \N$ we have 
\begin{gather*}
e^{\varphi}\#b\#e^{-\varphi}=b(1-\mu)
+S(\xiga^{p-l\delta-\ep}, \bar g)+S_{\rho,1/2}(\xiga^{-N}),\;\; (l-1)\delta\leq 1/2.
\end{gather*}
\end{lem}
\begin{proof} 
Note that $J_{\al,k}(b)\in S(w^{2l\delta-2\delta|\al+\be|-(1-2\delta)(\tilde\al_{x_1}+\tilde\be_{x_1})}\xiga^{p+k\kappa-|\al|}, \bar g)$. If the power of $w$ is nonnegative, from $-k+k\kappa=-(1/2+\delta)k-k\ep$, we have $J_{\al,k}(b)\in S(\xiga^{p-1/2-\delta-k\ep},\bar g)$. If the power is nonpositive one has $J_{\al,k}(b)\in S(\xiga^{p-l\delta-k\ep},\bar g)$ from \eqref{eq:q:furoku}, hence the assertion.
\end{proof}
\begin{cor}
\label{cor:b:key:b}Assume $b\in S_{\rho,\delta}^{(s)}(\xiga)\cap S(w^{4\delta}\xiga, \munderbar g)$ in the Gevrey $3$ case and $b\in S_{1,0}^{(s)}(\xiga)\cap S(w^{6\delta}\xiga, \munderbar g)$ in the Gevrey $4$ case, then we have
\begin{gather*}
e^{\varphi}\#b\#e^{-\varphi}=b(1-\mu)+S(\xiga^{\kappa'-2\ep},\bar g)+S_{\rho,1/2}(1).
\end{gather*}
\end{cor}
\begin{proof}
It is enough to apply Lemma \ref{lem:b:key} with $l=2$ or $l=3$ and $p=1$. 
\end{proof}
\begin{lem}
\label{lem:add:key}If 
$b\in S(w^{2\delta}\xiga^{1+\kappa'}, \munderbar g)\cap S^{(s)}_{\rho,\delta}(\xiga^{1+\kappa'})$ verifies $\dif_{\xi_1}b\in S(w^{2\delta}\xiga^{\kappa'}, \bar g)$ then
$e^{\varphi}\#b\#e^{-\varphi}=b(1-\mu)+S_{\rho,1/2}(\xiga^{1/2+\kappa'-\ep})$. 
\end{lem}
\begin{proof}It suffices to repeat the proof of Lemma \ref{lem:b:key} since $1-2\delta\leq 1/2$.
\end{proof}
%
%
%
\subsection{Composition $e^{\varphi}\#q\#e^{-\varphi}$}

Denote $\hat\phi_1=\phi_1\xiga^{-1}$ and
\[
\varrho=\big(\phi_2^2\xiga^{-2}+\Sigma_{j=3}^{d+1}\psi_j^2+2\tilde w\hat\phi^2_1(1-\tilde w/2)+\ell \xiga^{-1}\big)^{1/2}
\]
so that $q=\varrho^2\xiga^2$, where we recall $\phi_2\xiga^{-1}=e_2\psi_2$. In the Gevrey $4$ case, the term $(-\psi\phi_1+\psi^2)\xiga^{-2}$ should be added inside the square root. From \eqref{eq:doto:1} there is $C>0$ such that
\begin{gather*}
\psi_2^2+(\Sigma_{j=1}^{d+1}\psi_j^4+\ell\xiga^{-1})\leq C(\phi_1^4\xiga^{-2}+\Sigma_{j=2}^{d+1}\phi_j^2+\ell\xiga)\xiga^{-2},\\
\psi_2^2+(\Sigma_{j=1}^{d+1}\psi_j^6+\ell^2\xiga^{-2})^{1/2}\leq C(\phi_1^2w^{2/3}+\Sigma_{j=2}^{d+1}\phi_j^2+\ell\xiga)\xiga^{-2}.
\end{gather*}
\begin{lem}
\label{lem:q:3:hyoka}{\rm(Gevrey 3)}
 $\dif_X^{\be}q\in S(\varrho^{2-|\be|}w^{\delta \be_{\xi_1}}\xiga^{2-|\be_{\xi}|}, \bar g)$ for $|\be|\leq 2$.
\end{lem}
\begin{proof}Denote $\psi'=(\psi_1,\ldots,\psi_{d+1})$ and $\psi''=(\psi_2,\ldots,\psi_{d+1})$  so that $|\hat\phi_1|\lesssim |\psi'|$. Note that Lemma \ref{lem:pre:a} shows that $\dif_{\xi_1}\psi''$ are linear combinations of $\psi'$. Since $\xiga^{-1} |\dif_{\xi_1}q|\lesssim |\hat\phi_1|w^{2/3}+|\psi'||\psi''|+\xiga^{-1}\lesssim \varrho w^{1/3}$ and $
\xiga^{|\be_{\xi}|-2} |\dif_X^{\be}q|\lesssim |\hat\phi_1|w^{2/3}+|\psi''|+\xiga^{-1}\lesssim \varrho$ with $|\be|=1$, the assertion for $|\be|=1$ is clear.  
Since $ |\dif_{\xi_1}^2q|\lesssim w^{2/3}+|\psi'|+\xiga^{-1}
\lesssim w^{2/3}$ and $
\xiga^{|\be_{\xi}|-1} |\dif_X^{\be}\dif_{\xi_1}q|\lesssim w^{2/3}+|\psi'|+\xiga^{-1}
\lesssim w^{2/3}$ with $|\be|=1$ we have the assertion for $|\be|=2$. 
\end{proof}
\begin{lem}
\label{lem:q:4:hyoka}{\rm(Gevrey $4$)} $\dif_X^{\be}q\in S(\varrho^{2-|\be|}w^{2\delta \be_{\xi_1}}\xiga^{2-|\be_{\xi}|}, \bar g)$ for $|\be|\leq 2$ unless $\be_{\xi_1}=2$. We have $
\dif_{\xi_1}^2q\in S(\varrho, \bar g)$ and $\dif_{\xi_1}^3q\in S(\xiga^{-1}w^{2\delta}, \bar g)$.
\end{lem}
\begin{proof} Since $\xiga^{-1} |\dif_{\xi_1}q|\lesssim |\hat\phi_1^3|+|\psi'||\psi''|+\xiga^{-1}\lesssim \varrho w^{1/2}$ and $
\xiga^{|\be_{\xi}|-2} |\dif_X^{\be}q|\lesssim |\hat\phi_1^3|+|\psi''|+\xiga^{-1}\lesssim \varrho$ with $|\be|=1$. Note that
\begin{gather*}
|\dif^2_{\xi_1}q|\lesssim |\psi'|^2+|\psi''|+\xiga^{-1}\lesssim \varrho,\;\; \xiga|\dif_{\xi_1}^3q|\lesssim |\psi'|+\xiga^{-1}\lesssim w^{1/2},
\\
\xiga^{-1+|\be_{\xi}|}|\dif_{\xi_1}\dif_X^{\be}q|\lesssim |\psi'|+\xiga^{-1}\lesssim w^{1/2}\;\;(|\be|=1)
\end{gather*}
and $\xiga^{-2+|\be_{\xi}|}|\dif_X^{\be}q|\lesssim 1\; (|\be|=2)$ which proves the assertion.
\end{proof}
\begin{lem}
\label{lem:q:key} The composition $e^{\varphi}\#q\#e^{-\varphi}$ is given by
%
\begin{gather*}
q(1-\mu)+S_{\rho,1/2}(\varrho w^{1/2}r^{-1}\xiga^{5/4-\ep})
+S_{\rho,1/2}(\xiga^{1-\ep})+S_{\rho, 1/2}(1).
\end{gather*}
\end{lem}
\begin{proof} Consider $J_{\al,k}(q)$ with the notation \eqref{eq:kigo}. Assume $1\leq|\hat\al+\hat\be|\leq 2$ and in addition $\hat\al_{\xi_1}+\hat\be_{\xi_1}\leq 1$ in the Gevrey $4$ case. From 
Lemmas \ref{lem:q:3:hyoka} and \ref{lem:q:4:hyoka} it follows from \eqref{eq:q:furoku} that $J_{\al,k}(q)$ belongs to
\[
S(\varrho^{2-|\hat\al+\hat\be|} w^{1/2}r^{-1}w^{(1-2\delta)(\hat\al_{\xi_1}+\hat\be_{\xi_1}-\tilde\al_{x_1}-\tilde\be_{x_1})-2\delta|\tilde\al+\tilde\be|+1/2}\xiga^{2+k\kappa-|\al|}, \bar g).
\]
Since $(\tilde\al_{x_1}+\tilde\be_{x_1})-(\hat\al_{\xi_1}+\hat\be_{\xi_1}) \leq \al_{x_1}+\tilde\be_{x_1}-\hat\be_{\xi_1}
=\tilde\be_{\xi_1}+\tilde\be_{x_1}
\leq |\tilde\be|=|\be|-|\hat\be|=|\al|-|\hat\be|$, we have $
J_{\al,k}(q)\in S(\varrho^{2-|\hat\al+\hat\be|} w^{1/2}r^{-1}\xiga^{7/4-|\hat\al+\hat\be|/2-k\ep}, \bar g)$ 
proving the assertion for $w^{1/2}r^{-1}\xiga^{3/4}\leq \xiga$. In the Gevrey $4$ case with $|\hat\al+\hat\be|=\hat\al_{\xi_1}+\hat\be_{\xi_1}=2$, noting $\tilde\al_{x_1}+\tilde\be_{x_1}\leq |\al|$, we see that $J_{\al,k}(q)\in S(\varrho w^{1/2}r^{-1}\xiga^{5/4-\ep}, \bar g)$ by \eqref{eq:q:furoku}. 
When $|\hat\al+\hat\be|\geq k$ ($k=3,4$) it follows from \eqref{eq:q:furoku} that $J_{\al,k}(q)\in S(\xiga^{1-\ep}, \bar g)$ in the Gevrey $k$ case. Therefore, it remains to examine the Gevrey $4$ case with $|\hat\al+\hat\be|=3$. From \eqref{eq:q:furoku} it  follows that $J_{\al,k}(q)\in S(\xiga^{1-\ep}, \bar g)$ if $|\hat\al|+k'+\hat\al_{x_1}+\hat\be_{x_1}\neq 0$, otherwise since $\tilde\be_{x_1}=\al_{\xi_1}$ as before we see that $J_{\al,k}(q)\in S(\xiga^{5/4-(|\al|-\al_{x_1}-\al_{\xi_1})/4-\ep k}, \bar g)$
for $|\hat\be|=3$. If $|\al|-\al_{x_1}-\al_{\xi_1}\neq 0$ then $J_{\al,k}(q)\in S(\xiga^{1-\ep}, \bar g)$. Otherwise, denoting $\al_{x_1}=\nu$ and $\al_{\xi_1}=\mu$ as before, 
we see that $J_{\al,k}(q)$ is a sum of terms such as $\dif_{\xi_1}^3q\dif_{\xi_1}^{\nu-3}\dif_{x_1}^{\mu}((\dif_{\xi_1}\varphi)^{\mu}(\dif_{x_1}\varphi)^{\nu})$, which is  in
 \[
S((wr^{-2})^{\nu}(w^{1/2}r^{-1})^{\mu}w^{1/2-(\nu-3)/2-\mu}\xiga^{2+k\kappa-k}, \bar g)\subset S(\xiga^{1-\ep k}, \bar g)
 \]
 by Lemmas \ref{lem:phi:precise} and \ref{lem:q:4:hyoka}. 
\end{proof}
Here, to proceed further, we shall make some observations on the inverse of $\op{e^{\pm\varphi}}$ (see \cite{Ni:JPDO}).

\begin{lem}
\label{lem:mu:hyoka}We have $
e^{-\varphi}\#e^{\varphi}=1-\mu$ with $\mu \in S_{\rho,1/2}(w^{2-4\delta}r^{-2}\xiga^{-2\delta-2\ep})$. 
In particular, $\mu \in  S_{\rho,1/2}(w^{1/2}r^{-1}\xiga^{-1/4-2\ep})\subset S_{\rho, 1/2}(\xiga^{-2\ep})$.
\end{lem}
\begin{proof}Noting $\Sigma_{|\al|=l}(\sigma D)^{\al}\dif_X^{\al}\varphi=0$, $l\geq 1$, it suffices to consider the term $\dif_X^{\al^1}\varphi\cdots\dif_X^{\al^k}\varphi\in S_{\rho,1/2}(w^{2-4\delta}r^{-2}w^{-2\delta|\al|-(1-2\delta)\al_{x_1}+4\delta}\xiga^{k\kappa-|\al_{\xi}|})$ with $k\geq 2$ by Lemma \ref{lem:phi:precise}. Since $(\al^{\sigma})_{x_1}+\al_{x_1}=\al_{\xi_1}+\al_{x_1}\leq |\al|$ we have
\[
(\sigma D)^{\al}(\dif_X^{\al^1}\varphi\cdots\dif_X^{\al^k}\varphi)\in S_{\rho,1/2}(w^{2-4\delta}r^{-2}\xiga^{2\delta|\al|+(1/2-\delta)|\al|-2\delta}\xiga^{-|\al|+k\kappa})
\]
contained in $S_{\rho,1/2}(w^{2-4\delta}r^{-2}\xiga^{-2\delta-2\ep})\subset S_{\rho,1/2}(w^{1/2}r^{-1}\xiga^{-1/4-2\ep})$.
\end{proof}
Since $\mu\in S_{\rho,1/2}(\xiga^{-2\ep})\subset S_{1/2,1/2}(\ga^{-2\ep})$ there exists $k\in S_{1/2,1/2}(1)$ for large $\ga$ such that $(1-\mu)\#(1+k)=(1+k)\#(1-\mu)=1$. 
\begin{lem}
\label{lem:gyaku} Let $w_{\al}$ $(\al\in \N^{2n})$ be $S_{1/2,1/2}$ admissible weights such that $w_{\al}w_{\be}\lesssim w_{\al+\be}$. Assume that $\dif_X^{\al}\mu\in S_{1/2,1/2}(w_{\al})$ for any $\al\in\N^{2n}$ with $|\al|\leq N$ then $\dif_X^{\al}k\in S_{1/2,1/2}(w_{\al})$ for any $\al\in\N^{2n}$ with $|\al|\leq N$.
\end{lem}
\begin{proof}Suppose $\dif_X^{\al}k\in S_{1/2,1/2}(w_{\al})$ for $|\al|\leq l$. Let $|\be|=l+1$. Note that
\[
\dif_X^{\be}k=\dif_X^{\be}\mu+\Sigma_{|\be''|\leq l} C_{\be'\be''}(\dif_X^{\be'}\mu)\#(\dif_X^{\be''}k)+\mu\#(\dif_X^{\be}k)
\]
hence $(1-\mu)\#(\dif_X^{\be}k)=\dif_X^{\be}\mu+\Sigma_{|\be''|\leq l} C_{\be'\be''}(\dif_X^{\be'}\mu)\#(\dif_X^{\be''}k)\in S_{1/2,1/2}(w_{\be})$ 
which proves that $\dif_X^{\be}k\in S_{1/2,1/2}(w_{\be})$.
\end{proof}
\begin{cor}
\label{cor:k:kuwasii}
We have $k\in S_{\rho,1/2}(w^{1/2}r^{-1}\xiga^{-1/4-2\ep})$.
\end{cor}
Let us summarize what we have proved so far so that they can be easily applied in what follows.
\begin{pro}
\label{pro:xi:0}We have $
(1+k)\#e^{\varphi}\#(\tau-\la_i)\#e^{-\varphi}
=\tau-\tilde\la_i$ 
where
\begin{gather*}
\tilde\la_i=\tilde\phi_1+i\te c_i\xiga^{\kappa'}
+S_{\rho,1/2}(\xiga^{\kappa'-\ep})
+S_{\rho,1/2}(1)
\end{gather*}
and $(1+k)\#e^{\varphi}\#(\tilde\phi_1+ik_2\xiga^{\kappa}\psi_{d+1}^2)\#e^{-\varphi}=b$ 
where  
\begin{gather*}
b=\tilde\phi_1+2i\{\phi_1,\psi_2\}wr^{-2}\xiga^{\kappa}+ik_2\xiga^{\kappa}\psi_{d+1}^2+S_{\rho,1/2}(wr^{-2}\xiga^{\kappa-\ep})\\
+S_{\rho,1/2}(w^{1/2}r^{-1}\xiga^{1/4-\ep})
+S_{\rho,1/2}(\xiga^{\kappa'-\ep})
+S_{\rho,1/2}(1).
\end{gather*}
\end{pro}
\begin{proof} From Lemma \ref{lem:pre:a} we have $\{\tau-\phi_1,\psi_j\}\in S( w^{2/3}, \bar g)$ for $1\leq j\leq d+1$ which implies  $\{\tau-\phi_1,\varphi\}\in S(w^{-1/3}\xiga^{\kappa}, \bar g)\subset S(\xiga^{1/6+\kappa}, \bar g)=S(\xiga^{\kappa'-2\ep}, \bar g)$ in the Gevrey $3$ case. In the Gevrey $4$ case, thanks to \eqref{eq:G4:bunkai} and Lemma \ref{lem:pre:a} it follows that $\{\tau-\phi_1,\varphi\}\in S(\xiga^{\kappa}, \bar g)=S(\xiga^{\kappa'-2\ep},\bar g)$. Note that
\[
e^{\varphi}\#\tau\#e^{-\varphi}=(1-\mu)\tau-((\dif_t\varphi e^{\varphi})\#e^{-\varphi}+e^{\varphi}\#(\dif_t\varphi e^{-\varphi}))/(2i).
\]
Since $\dif_{\xi_1}^2\phi_1\in S(w^{2\delta}\xiga^{-1}, \bar g)$  
by \eqref{eq:phi:1:katati} in the Gevrey $4$ case, one can apply Lemma \ref{lem:b:key:key} with $b=\phi_i$ to obtain 
\begin{gather*}
e^{\varphi}\#(\tau-\phi_1)\#e^{-\varphi}=(1-\mu)(\tau-\phi_1)
-((\{\tau-\phi_1,\varphi\}e^{\varphi})\#e^{-\varphi}\\
+e^{\varphi}\#(\{\tau-\phi_1\}e^{-\varphi}))/(2i)+S_{\rho,1/2}(\xiga^{\kappa'-\ep})+S_{\rho,1/2}(1)\\
=(1-\mu)(\tau-\phi_1)+S_{\rho,1/2}(\xiga^{\kappa'-\ep})+S_{\rho,1/2}(1).
\end{gather*}
As for $\tilde w\phi_1$, we apply  Corollary \ref{cor:b:key:b}. To estimate $\{\tau-\phi_1,\mu\}$, taking Corollary \ref{cor:matome} into account it suffices to consider $\{\tau-\phi_1, (\sigma \dif_X)^{\al}(\dif_X^{\al^1}\varphi\cdots \dif_X^{\al^k}\varphi)\}$, which can be written as $\Sigma_{l}(\sigma \dif_X)^{\al}(\Pi_{j\neq l}\dif_X^{\al^j}\varphi)\dif_X^{\al^l}\{\tau-\phi_1,\varphi\}+r_l$ where $r_l$ consists of a sum of terms that contains $\dif_X^{\ga}\phi_1$ with $|\ga|\geq 2$ as a product factor. Applying Lemma \ref{lem:yobi} to $J_{\al+t,k}(\phi_1)$ with $|t|=1$ we conclude $r_l\in S(\xiga^{\kappa'-\ep})$, and hence $\{\tau-\phi_1,\mu\}\in S_{\rho,1/2}(\xiga^{\kappa'-\ep})$. This proves $(1-\mu)(\tau-\phi_1)=(1-\mu)\#(\tau-\phi_1)+S_{\rho,1/2}(\xiga^{\kappa'-\ep})$ thus the first assertion. Turn to the second assertion. Noting that $\{\phi_1,\varphi\}\in S_{\rho,1/2}(wr^{-2}\xiga^{\kappa})$ and Corollary \ref{cor:b:key:b}, it follows from Lemma \ref{lem:b:key:key} and Proposition \ref{thm:matome:b} that
\begin{equation}
\label{eq:b:matome:b}
\begin{split}
&e^{\varphi}\#\tilde\phi_1\#e^{-\varphi}=\tilde\phi_1(1-\mu)+i\{\phi_1,\varphi\}\\
&+S_{\rho,1/2}(wr^{-2}\xiga^{\kappa-\ep})
+S_{\rho,1/2}(\xiga^{\kappa'-2\ep})+S_{\rho,1/2}(1).
\end{split}
\end{equation}
Since $\tilde\phi_1(1-\mu)=(1-\mu)\#\tilde\phi_1+S_{\rho,1/2}(w^{1/2}r^{-1}\xiga^{1/4-2\ep})$ and  
\[
\{\phi_1,\varphi\}=2wr^{-2}\{\phi_1, \psi_2\}\xiga^{\kappa}+S(w^{1-2\delta}r^{-1}\xiga^{\kappa}, \bar g)
\]
the assertion follows from \eqref{eq:b:matome:b} for $w^{1-2\delta}r^{-1}\xiga^{\kappa}\leq w^{1/2}r^{-1}\xiga^{1/4-\ep}$. 
\end{proof}
\begin{pro}
\label{pro:q} The composition $(1+k)\#e^{\varphi}\#(q+q_1')\#e^{-\varphi}$ is given by 
\begin{gather*}
\tilde q=q+\te (t-\tka) q_1+q_2
+S_{\rho,1/2}(\varrho w^{1/2}r^{-1}\xiga^{5/4-\ep})
+S_{\rho,1/2}(\xiga^{1-\ep})
\end{gather*}
with $q_i\in S_{\rho,1/2}(\varrho\xiga^{1+\kappa'-(i-1)\ep})$.
\end{pro}
\begin{proof} We have $
(1-\mu)q=(1-\mu)\#q+S_{\rho,1/2}(\varrho w^{1/2}r^{-1}\xiga^{5/4-2\ep})$ by Lemma \ref{lem:mu:hyoka} and $\dif_X^{\be}q\in S(\varrho \xiga^{2-|\be_{\xi}|}, \bar g)$ for $|\be|=1$. Since $q_1'\in S(w^{2\delta}\xiga^{1+\kappa'}, \munderbar g)\cap S^{(s)}_{\rho,\delta}(\xiga^{1+\kappa'})$ we obtain $e^{\varphi}\#q'_1\#e^{-\varphi}=(1-\mu)q_1'+S_{\rho,1/2}(\xiga^{1/2+\kappa'+\ep})$ from Lemma \ref{lem:add:key}. Noting that $(1-\mu)q_1'=(1-\mu)\#q_1'+S_{\rho,1/2}(\xiga^{1/2+\kappa'-2\ep})$ and recalling that $q_1'=\te (t-\tka) S(\varrho \xiga^{1+\kappa'}, \bar g)+S(\varrho \xiga^{1+\kappa},\bar g)$, the assertion follows from Lemma \ref{lem:q:key}.
\end{proof}
Noting \eqref{eq:amari} we see that $(1+k)\#e^{\varphi}\#T\#p\# \tilde T\#e^{-\varphi}$ is written as
\[
-(\tau-\tilde\la_2)\#(\tau-\tilde\la_1)
-2 b\#(\tau-\tilde\la_1)+ \tilde q+S_{\rho, 1/2}(1)
\]
where $\tilde q$ is given by Proposition \ref{pro:q}. Denote
\begin{gather*}
\La_i=D_t-\op{\tilde\la_i},\quad
B=\op{b},\quad
Q=\op{\tilde q}.
\end{gather*}
\begin{lem}
\label{lem:q:kokan}There exist $q'\in S_{\rho,1/2}(\varrho^2\xiga^{2+\kappa'-\ep})$,  $q''\in S_{\rho,1/2}(\varrho \xiga^{3/2+\ep})$  such that
\[
[\La_1,Q]=\op{q'+q''+S_{\rho,1/2}(\xiga)}. 
\]
\end{lem}
\begin{proof}
Taking $e^{\phi}\#(1+k)\#e^{-\phi}=1$ into account we have
\begin{gather*}
[\La_1,  Q]
=\op{(1+k)\#e^{-\varphi}\#((\tau-\la_1)\#\hat q-\hat q\#(\tau-\la_1))\#e^{\varphi}},\;\;\hat q=q+q_1'.
\end{gather*}
By Lemma \ref{lem:pre:a} we have $\{\tau-\phi_1, q'_1\}\in S(w^{2\delta}\xiga^{1+\kappa'},\bar g)$ then we see that $(\tau-\la_1)\#q_1'-q_1'\#(\tau-\la_1)\in S_{\rho,1/2}(w^{2\delta}\xiga^{1+\kappa'})+S_{\rho,1/2}(\xiga^{2\kappa'})$. On the other hand, 
We have $(\tau-\la_1)\#q-q\#(\tau-\la_1)=\{\tau-\la_1,q\}/i+S_{\rho,1/2}(\xiga)$. 
Taking $\tilde w\phi_1^2\in S(w^2\xiga^2,\munderbar g)$ (in the Gevrey $3$ case) into account, it follows from Lemma \ref{lem:pre:a} that
\[
\{\tau-\phi_1,q\}\in S(\varrho^{4/3}\xiga^2, \bar g)\quad \text{or}\quad  \{\tau-\phi_1,q\}\in S(\varrho^{3/2}\xiga^2,\bar g)
\]
according to the Gevry 3 or 4 case, and $\{\xiga^{\kappa'}, q\}\in S(\varrho\xiga^{1+\kappa'},\bar g)$ is easy. From \eqref{eq:doto:1} there is $C>0$ such that
\[
\xiga^{-1}\leq r^2=\psi_2^2+w^2\leq C\varrho^2
\]
then $\varrho^{4/3}\xiga^{2}\leq \rho^2\xiga^{2+1/3}\leq \rho^2\xiga^{2+\kappa'-\ep}$ and $\varrho^{3/2}\xiga^2\leq \varrho^2\xiga^{2+1/4}=\varrho^2\xiga^{2+\kappa'-\ep}$ in the Gevrey 3 and $4$ case respectively. Since it is also clear that $w^{2\delta}\xiga^{1+\kappa'}\leq \varrho \xiga^{3/2+\kappa'-\delta}=\varrho\xiga^{3/2+\ep}$ we finish the proof.
\end{proof}
%

\section{Energy estimates and existence of solution}

\begin{pro}
\label{pro:ene:id} Let $\hat P=-\La_2\La_1+2B\La_1+Q$ then we have
\begin{eqnarray*}
2{\mathsf{Im}}(\hat P v, \La_1 v)=\frac{d}{dt}(\|\La_1 v\|^2+((\mathsf{Re}\,Q) v,v))
+2((\mathsf{Im}B) \Lambda_1 v, \Lambda_1 v)\\
+2(\op{\mathsf{Im}\tilde \la_2}\La_1 v, \La_1 v)
+2{\mathsf{Re}}(\La_1 v, (\mathsf{Im}Q)v)\\
+{\mathsf{Im}}([\La_1, Q]v,v)
+2{\mathsf{Re}}(Q v, \op{\mathsf{Im}\tilde\lambda_1} v).
\end{eqnarray*}
\end{pro}
To check this, it is enough to note
\begin{gather*}
2{\mathsf{Im}}(Q v, \La_1 v)=\frac{d}{dt}(({\mathsf{Re}}\,Q) v, v)+2{\mathsf{Re}}(Q v,\op{{\mathsf{Im}}\tilde\la_1}v)\\
+2{\mathsf{Re}}(\La_1 v, ({\mathsf{Im}}Q)v)+{\mathsf{Im}}([\La_1, Q]v,v).
\end{gather*}
Since $g_{\rho,1/2}/g^{\sigma}_{\rho,1/2}\leq \ga^{-(\rho-\delta)}$ and $g_{\rho,1/2}\leq g_{1/2,1/2}$, we can apply the rsults in \cite[Appendix]{Ni:JPDO}. Since $\tilde w\phi_1^2\xiga^{-2}\in S(w^2, \munderbar g)$, following the arguments in \cite[Appendix]{Ni:JPDO} we see that $\varrho\in S_{1/2,1/2}(\varrho)$ is $S_{1/2,1/2}$ admissible weight and there is $\ell_0>0$, independent of $\ga$ such that for $\ell\geq \ell_0$ there exists $\tilde\varrho\in S_{1/2,1/2}(\varrho^{-1})$ verifying $\varrho\#\tilde \varrho=\tilde\varrho\#\varrho=1$, and that we have
\[
(\op{\varrho^2}v, v)\geq \|\op{\varrho} v\|^2/2,\quad \ell\geq \ell_0.
\]
Since ${\mathsf{Im}}\tilde\la_i=\te c_i\xiga^{\kappa'}+S_{\rho,1/2}(\xiga^{\kappa'-\ep})$ by Proposition \ref{pro:xi:0}, there is $c>0$ such that
\begin{equation}
\label{eq:la:sita}
(\op{{\mathsf{Im}}\tilde\la_2}\La_1v,\La_1 v)\geq c\,\te \|\lr{D}_{\ga}^{\kappa'/2}\La_1 v\|^2,\quad \ga\geq \ga_1(\theta).
\end{equation}
Noting that $\varrho\xiga^{1+\kappa'}\leq \varrho^2\xiga^{2-\kappa}$ and $\varrho w^{1/2}r^{-1}\xiga^{5/4-\ep}\leq \varrho^2\xiga^{2-\ep}$ one has ${\tilde q}=q+S_{\rho,1/2}(\varrho^2\xiga^{2-\ep})+S_{\rho,1/2}(\xiga^{1-\ep})$ from Proposition \ref{pro:q}, and hence
\[
({\mathsf{Im}}\tilde\la_i)\#\tilde q=\te c_i\xiga^{\kappa'}q+S_{\rho,1/2}(\varrho^2\xiga^{2+\kappa'-\ep})+S_{\rho,1/2}(\xiga^{1+\kappa'-\ep}).
\]
Since $q=\varrho^2\xiga^2$ we have $(\op{c_i\xiga^{\kappa'}q}, v, v)\geq \|\op{c_i^{1/2}\varrho\xiga^{1+\kappa'/2}} v\|^2/2$ which is bounded from below by $c\big(\|\op{\varrho\xiga^{1+\kappa'/2}}v\|^2+\|\op{\xiga^{1/2+\kappa'/2}}v\|^2\big)$ ($c>0$) for $\varrho\geq \xiga^{-1/2}$, which gives 
\begin{equation}
\label{eq:tilde:Q}
\begin{split}
(({\mathsf{Re}}\,Q) v,v)&\geq c(\|\op{\varrho \xiga }v\|^2+\|\lr{D}_{\!\ga}^{1/2}v\|^2),\\
{\mathsf{Re}}(Q v, \op{{\mathsf{Im}}\tilde\la_1}v)&\geq c\,\te(\|\op{\varrho\xiga^{1+\kappa'/2}}v\|^2+\|\lr{D}_{\!\ga}^{1/2+\kappa'/2}v\|^2)
\end{split}
\end{equation}
for $\ga\geq \ga_0$. From Proposition \ref{pro:q} it follows that
\begin{gather*}
{\mathsf{Im}}\,\tilde q=\theta(t-\tka)S_{\rho,1/2}(\varrho \xiga^{1+\kappa'})+S_{\rho,1/2}(\varrho \xiga^{1+\kappa'-\ep})\\
+ S_{\rho,1/2}(\varrho w^{1/2}r^{-1}\xiga^{5/4-\ep})
+S_{\rho,1/2}(\xiga^{1-\ep}).
\end{gather*}
Since $rw^{-1/2}\xiga^{1-\kappa/2}\leq \varrho\xiga^{5/4-\kappa/2}=\varrho\xiga^{1+\kappa'/2}$ and $\kappa'+\kappa=1/2$, one has 
\begin{equation}
\label{eq:xi:1}
\begin{split}
\xiga=(w^{1/2}r^{-1}\xiga^{\kappa/2})(rw^{-1/2}\xiga^{1-\kappa/2})\\
\in S_{\rho,1/2}(w^{1/2}r^{-1}\xiga^{\kappa/2}\varrho \xiga^{1+\kappa'/2}).
\end{split}
\end{equation}
Taking \eqref{eq:xi:1} into account, we obtain
%
\begin{equation}
\label{eq:Q:La}
\begin{split}
|(\La_1 v,({\mathsf{Im}}Q)v)|\leq C(\ga^{-\ep}+\te \tka)\|\op{\varrho \xiga^{1+\kappa'/2}}v\|^2\\
+C(\ga^{-\ep}+\te \tka)\|\lr{D}_{\ga}^{\kappa'/2}\La_1 v\|^2
+C\ga^{-\ep}\|\op{w^{1/2}r^{-1}\xiga^{\kappa/2}}\La_1v\|^2
\end{split}
\end{equation}
for $t\in[-\tka, \tka]$. 
Since $\{\phi_1,\psi_2\}=\chi'(x_1)+S(w^{2/3}, \bar g)$ or $\{\phi_1,\psi_2\}=\chi'(x_1)+S(r, \bar g)$ by \eqref{eq:phi:1:a} or \eqref{eq:G4:bunkai} in the Gevrey 3 or 4 case respectively, it follows from Proposition \ref{pro:xi:0} that 
\begin{gather*}
{\mathsf{Im}}\,b=2wr^{-2}\xiga^{\kappa}\chi'(x_1)+k_2\xiga^{\kappa}\psi^2_{d+1}
+S_{\rho,1/2}(wr^{-2}\xiga^{\kappa-\ep})\\
+S_{\rho,1/2}(w^{1/2}r^{-1}\xiga^{1/4-\ep})
+S_{\rho,1/2}(\xiga^{\kappa'-\ep}).
\end{gather*}
Note that for any $\ep_1>0$ there exist $c, \ep_2>0$ such that $|\psi_2|\geq c$ if $|x_1|\geq \ep_1$ and $|t|\leq \ep_2$ so that $r\geq c'>0$ there. Then choosing $k_2>0$ suitably we have $wr^{-2}\chi'(x_1)+k_2\psi_{d+1}^2\geq cwr^{-2}$ with some $c>0$, proving that
\[
(\op{wr^{-2}\xiga^{\kappa}\chi'(x_1)+k_2\xiga^{\kappa}\psi_{d+1}^2}u,u)\geq c_1\|\op{w^{1/2}r^{-1}\xiga^{\kappa/2}}u\|^2
\]
with some $c_1>0$. Therefore we obtain
\begin{equation}
\label{eq:B:sita}
\begin{split}
2( ({\mathsf{Im}}B) \Lambda_1 v,\La_1 v)\geq c\|\op{w^{1/2}r^{-1}\xiga^{\kappa/2}}\La_1v\|^2\\
-C\ga^{-\ep}\|\lr{D}_{\ga}^{\kappa'/2}\La_1v\|^2,\quad \ga\geq \ga_0.
\end{split}
\end{equation}
Finally, it follows from Lemma \ref{lem:q:kokan} that
\begin{equation}
\label{eq:q:kpkan}
\begin{split}
|([\La_1, Q]v, v)|\leq C(\ga^{-\ep}\|\op{\varrho\xiga^{1+\kappa'/2}} v\|^2
+\ga^{-\ep}\|\lr{D}_{\ga}^{1/2} v\|^2).
\end{split}
\end{equation}
For any lower-order term $a\in S_{\rho,1/2}(\xiga)$, we deduce from \eqref{eq:xi:1} that
\begin{equation}
\label{eq:teikai}
\begin{split}
|(\op{a}v, \La_1 v)|\leq C\|\op{\varrho\xiga^{1+\kappa'/2}}v\|\|\op{w^{1/2}r^{-1}\xiga^{\kappa/2}}\La_1v\|\\
\leq C\te^{-1}\,(\te\|\op{r^{-1}w^{1/2}\xiga^{\kappa/2}}\La_1v\|^2+\te \|\op{\varrho\xiga^{1+\kappa'/2}}v\|^2).
\end{split}
\end{equation}
Since $\ell$ is independent of $\te$ one can controll $|(\ell \lr{D}_{\ga} v, \La_1v)|$ by \eqref{eq:teikai} choosing $\theta$ large. 
%
%
We now first fix $\ell$ then fix $\te$ then $\tka$ and finally fix $\ga$. Then from \eqref{eq:la:sita}, \eqref{eq:tilde:Q}, \eqref{eq:Q:La}, \eqref{eq:B:sita}, \eqref{eq:q:kpkan} and  \eqref{eq:teikai} one can find $c>0,\tka>0, \ga_0>0$ such that
 \begin{gather*}
 2{\mathsf{Im}}(\hat P v, \La_1 v)\geq \frac{d}{dt}(\|\La_1 v\|^2+((\mathsf{Re}\,Q) v,v))+c\,\|\lr{D}_{\ga}^{\kappa'/2}\La_1 v\|^2\\
 +c\|\op{w^{1/2}r^{-1}\xiga^{\kappa/2}}\La_1v\|^2+c\,\|\op{\varrho\xiga^{1+\kappa'/2}}v\|^2+c\|\lr{D}_{\!\ga}^{1/2+\kappa'/2}v\|^2
 \end{gather*}
holds for $0\leq t\leq \tka$  and $\ga\geq\ga_0$. 

Following \cite[Chapter 7]{Ni:book} one can prove the existence of the solution operator $\hat G$ of $\hat Pu=f$ with finite propagation speed, that is 
\begin{gather*}
\hat P\hat Gf=f,\quad \sum_{j=0}^1\|\lr{D}^{l+\kappa'-j}D_t^j \hat Gf\|\leq C_s\int ^t\|\lr{D}^lf(x_0)\|^2dx_0,\quad l\in\R
\end{gather*}
for any $f\in C^0([-\tka,\tka];H^l)$ vanishing in $t\leq 0$, and moreover for any $h_i\in S^0_{1,0}(1)$ with compact supports such that ${\rm supp}h_1\cap{\rm supp}h_2=\emptyset$ there is $\delta=\delta(h_i)>0$ so that we have for any $p, q\in \R$
\begin{equation}
\label{eq:G:hyoka}
\sum_{j=0}^1\|D_t^jh_2\hat Gh_1f\|_{p-j}\leq C_{pq}\int^t\|f(x_0)\|_qdx_0,\quad |t|\leq \delta
\end{equation}
for any $f\in C^0([-\delta,\delta];H^{q})$ vanishing in $t\leq 0$. 

Returning to the original operator, we have $
\op{e^{\varphi}}T\op{P+R}\tilde T\op{e^{-\varphi}}\hat Gf=f$, that is
\[
\op{P+R}\tilde T\op{e^{-\varphi}}\hat G\op{e^{\varphi}} Tf=f
\]
where $R=a_1\tau+a_2$ with $a_j\in S^{j}_{1,0}$ which vanishes near $\bro$. By a compactness argument, one can find a finite number of open conic neighborhoods $V_i$ of $(0, \xi^{(i)})$ and solution operators $G^{(i)}$ such that $\cup_iV_i\supset \{0\}\times (\R^n\setminus\{0\})$ and
\[
\op{P+R^{(i)}}G^{(i)}f=f,\quad G^{(i)}=\tilde T\op{e^{-\varphi^{(i)}}}\hat G^{(i)}\op{e^{\varphi^{(i)}}} T
\]
where $R^{(i)}=a_1^i\tau+a_2^i$ with $a_j^i=0$ in $U_i$ where $V_i\Subset U_i$. Let $\{\al^i(x,\xi)\}$ be a partition of unity subordinate to $\{V_i\}$ such that $\Sigma \al^i=\al(x)$ where $\al(x)$ is $1$ in a small neighborhood of the origin. Denote
\[
G=\sum_i G^{(i)}\op{\al^i}
\]
so that $\op{P}G f=\al(x)f-Rf$ where $R=\sum_i\op{R^{(i)}}G^{(i)}\op{\al^i}f$. Here recall that $\varphi^{(i)}=-\xiga^{\kappa}(\phi^{(i)}+k_3\chi^{(i)})$ and we can assume that $\phi^{(i)}+k_3\chi^{(i)}\geq c>0$ on $V_i$ and $\phi^{(i)}+k_3\chi^{(i)}\leq -c$ outside $U_i$. Then, writing $R^{(i)}\#\tilde T=\tilde T\#(\tilde R^{(i)}+S_{0,0}^{(s)}(e^{-c\xiga^{1/s}}))$ and $T\#\al^{(i)}=(\tilde \al^{(i)}+S_{0,0}^{(s)}(e^{-c\xiga^{1/s}}))\# T$ we  apply Proposition \ref{pro:0:0} to $\tilde R^{(i)}\#e^{-\varphi^{(i)}}$ and $e^{\varphi^{(i)}}\#\tilde\al^{(i)}$. As a result, thanks to \eqref{eq:G:hyoka} one concludes that 
\begin{equation}
\label{eq:R:hyoka}
\|TRf(t)\|_l\leq C_l\int^t\|Tf(t')\|_ldt',\quad 0\leq t\leq \bar\delta=\bar\delta(\{U_i\}, \{V_i\}).
\end{equation}
Multiply \eqref{eq:R:hyoka} by $e^{-\theta' t}$ and integrate from $0$ to $t$ one obtains
\[
\int_0^te^{-\theta' t'}\|TRf(t')\|_ldt'\leq \frac{C_l}{\theta'}\int_0^te^{-\theta' t'}\|Tf(t')\|_ldt'.
\]
Choose $\theta'=\theta'_l$ such that $C_l/\theta'<1/2$ then $Sf=\Sigma_{k=0}^{\infty}R^kf$ converges in the weighted $L^1([0, \bar\delta]; H^l)$ with weight $e^{-\theta' t}e^{\theta\lr{D}_{\!\ga}^{\kappa'}(\tka-t)}$ with $\kappa'$ given in \eqref{eq:k:k}, where $\ep>0$ can be chosen arbitrarily small. Let $\be(x)$ be $1$ near the origin such that $\be(x)\al(x)=\be(x)$, then since $\be \op{P} Gf=\be(1-R)f$ we conclude that $u=GSf$ satisfies $\op{P}u=f$ in $\{\be(x)\neq 0\}$.

%

\section{Appendix}

In this appendix, we prove a composition formula of pseudodifferential operators acting on Gevrey spaces. It is somewhat less precise than that of \cite{NTa}, but in return, the proof does not make use of almost analytic extension and is much easier to apply.

\subsection{Oscillatory integral of Gevrey symbol}
\label{sec:sindo}

\begin{definition}
\label{dfn:Gev}\rm We say that $f(x)\in C^{\infty}(\R^{n})$ belongs to $G^s(\R^n)$, the (global) Gevrey class $s$, if there exist $C>0, A>0$ such that
\[
|D^{\al}f(x)|\leq CA^{|\al|}|\al|!^s,\quad x\in \R^n,\;\;\alpha\in\N^{n}
\]
holds. We denote $G_0^s(\R^n)=C_0^{\infty}\cap G^s(\R^n)$.
\end{definition}
Recall $\xiga^2=\ga^2+|\xi|^2$ and $ \lr{\xi}_1=\lr{\xi}$ 
 where $\ga\geq 1$ is a positive parameter.
We introduce a symbol class for which we define oscillatory integrals.
\begin{definition}
\label{dfn:G:symbol}\rm Let $m=m(x, \xi, \ga)$ be a positive function and $0\leq \delta<1$, $1<s$. We say that $a(x, \xi, y; \ga)\in C^{\infty}(\R^{3n})$ belongs to ${\mathcal { A}^{(s)}_{\delta}}(m)$ if there are $C,  A>0$ independent of $\ga\geq 1$ such that
\[
|\dif_{x, y}^{\be}\dif_{\xi}^{\al}a(x, \xi, y; \ga)|\leq CA^{|\al|+|\be|}(|\al|+|\be|)!^s
(|\be|^{\delta s/(1-\delta)}+\xiga^{\delta})^{|\be|}
 m(x, \xi; \ga)
\]
for all $\al\in\N^n, \be\in\N^{2n}$. We often write just $a(x,\xi, y)$ or $m(x, \xi)$ dropping $\ga$.
By abuse of notation we denote by the same ${\mathcal { A}^{(s)}_{\delta}}(m)$ the set of all $a(x, \xi; \ga)\in C^{\infty}(\R^{2n})$ satisfying the above estimates.
\end{definition}
Assume that 
$a(x, \xi, y)\in {\mathcal A^{(s)}_{\delta}}(e^{c\xiga^{{\kappa}}})$ $(c>0)$ with $1-\delta>s\kappa$.
Let $\chi(t)\in G^s_0({\mathbb R}^n)$ be even such that $\chi(t)=1$ in a neighborhood of $0$
 and set $\chi_{\epsilon}(y)=\chi(\ep y)$, $\chi_{\ep}(\eta)=\chi(\ep\eta)$. Let $\rho(t)\in G^s(\R)$ be such that $\rho(t)=0$ for $|t|\leq 1/2$ and $\rho(t)=1$ for $|t|\geq 1$ and set $\rho_{\ga}(\eta)=\rho(\ga^{-1}\eta)$, $\rho_{\ga}^c(\eta)=1-\rho_{\ga}(\eta)$. We define ${\mathcal Op}(a)u(x)$ for $u\in G^{s/(1-\delta)}(\R^n)$ by the oscillatory integral 
\begin{equation}
\label{eq:Op:teigi}
\begin{split}
{\mathcal Op}(a)u(x)=(2\pi)^{-n}\lim_{\ep\to 0}\int e^{i(x-y)\eta}\chi_{\ep}(y-x)\chi_{\ep}(\eta)a(x, \eta, y)u(y)dyd\eta\\
=(2\pi)^{-n}\lim_{\ep\to 0}\int e^{-iy\eta}\chi_{\ep}(y)\chi_{\ep}(\eta)a(x, \eta, y+x)u(y+x)dyd\eta.
\end{split}
\end{equation}
After integration by parts ${\mathcal Op}(\rho_{\ga}a)u(x)$ yields
\[
\int e^{-iy\eta}\lr{D_y}^{2N}\lr{\eta}^{-2N}\lr{D_{\eta}}^{2\ell}\lr{y}^{-2\ell}\chi_{\epsilon}(y)\chi_{\epsilon}(\eta)\rho_{\ga}a(x, \eta, y+x)u(y+x)dyd\eta.
\]
Since  $s/(1-\delta)=s+s\delta/(1-\delta)$ and $\lr{\eta}_{\!\ga}\leq 3\lr{\eta}$ if $\rho_{\ga}\neq 0$ the integrand is bounded uniformly in $\epsilon>0$ by ($C, A$ may change line by line but not depend on $N$)
\begin{equation}
\label{eq:kurikaesi}
\begin{split}
CA^{2N} N^{2Ns}(N^{s\delta/(1-\delta)}+\lr{\eta}_{\!\ga}^{\delta})^{2N}\lr{y}^{-2\ell}\lr{\eta}^{-2N}e^{c\lr{\eta}_{\!\ga}^{\kappa}}\\
\leq C\lr{y}^{-2\ell}
 \Big(\frac{rAN^s}{\lr{\eta}^{1-\delta}}\Big)^{2N}\Big(\frac{N^{s\delta/(1-\delta)}}{r\lr{\eta}^{\delta}}+\frac{3^{\delta}}{r}\Big)^{2N}e^{c\lr{\eta}_{\!\ga}^{{ \kappa}}}
 \end{split}
\end{equation}
for any $r>0$. Choosing $r, \bar C>0$ suitably and take 
the maximal $N=N(\eta)\in \N$ with $N^s\leq \lr{\eta}^{1-\delta}/(\bar C A)$ one can find $c'>0$ so that 
\[
\Big(\frac{rAN^s}{\lr{\eta}^{1-\delta}}\Big)^{2N}\Big(\frac{N^{s\delta/(1-\delta)}}{r\lr{\eta}^{\delta}}+\frac{3^{\delta}}{r}\Big)^{2N}\leq Ce^{-c'\lr{\eta}^{(1-\delta)/s}}.
\]
Since ${ \kappa}<(1-\delta)/s$, the integrand is bounded by $C\lr{y}^{-2\ell}e^{-c''\lr{\eta}^{(1-\delta)/s}}$. Noting that $\dif_{y, \eta}^{\al}\chi_{\epsilon}\to 0$ as $\epsilon\to 0$ if $|\al|\geq 1$ we see that ${\mathcal Op}(\rho_{\ga}a)u(x)$ is 
\[
\int e^{-iy\eta}\lr{D_y}^{2N}\lr{\eta}^{-2N}\lr{D_{\eta}}^{2\ell}\lr{y}^{-2\ell}\rho_{\ga}(\eta)a(x, \eta, y+x)u(y+x)dyd\eta.
\]
Repeating a similar argument for ${\mathcal Op}(\rho_{\ga}^ca)u(x)$ we see that \eqref{eq:Op:teigi} is equal to
\[
\int e^{-iy\eta}\lr{D_y}^{2N}\lr{\eta}^{-2N}\lr{D_{\eta}^{2\ell}}\lr{y}^{-2\ell}a(x, \eta, y+x)u(y+x)dyd\eta 
\]
which is independent of the choice of $\chi$. Next, consider $\dif_x^{\be}{\mathcal Op}(a)u(x)$;
\[
\int e^{-iy\eta}\lr{D_y}^{2N}\lr{\eta}^{-2N}\lr{D_{\eta}}^{2\ell}\lr{y}^{-2\ell}\chi_{\epsilon}(y)\chi_{\epsilon}(\eta)\dif_x^{\be}(a(x, \eta, y+x)u(y+x))dyd\eta.
\]
Here we remark the following elementary inequality.
\begin{lem}
\label{lem:kantan}
Let $A, B\geq 0$. Then there exists $C>0$ independent of $n, m\in\N$, $A$, $B$ such that
\begin{align*}
(A+(n+m)^sB)^{n+m}
\leq C^{n+m}(A+n^sB)^n(A+m^sB)^m.
\end{align*}
\end{lem}
Taking Lemma \ref{lem:kantan} into account, the integrand is bounded by
\begin{gather*}
CA^{2N+|\be|}(2N+|\be|)!^s((2N+|\be|)^{s\delta/(1-\delta)}+\lr{\eta}_{\!\ga}^{\delta})^{2N+|\be|}\lr{y}^{-2\ell}\lr{\eta}^{-2N}e^{c\lr{\eta}_{\ga}^{{\kappa}}}
\\
\leq CA^{2N+|\be|}|\be|!^s(|\be|^{s\delta/(1-\delta)}+\lr{\eta}_{\!\ga}^{\delta})^{|\be|}\\
\times N^{2Ns}(N^{s\delta/(1-\delta)}+\lr{\eta}_{\!\ga}^{\delta})^{2N}\lr{y}^{-2\ell}\lr{\eta}^{-2N}e^{c\lr{\eta}_{\!\ga}^{\kappa}}.
\end{gather*}
For any $\ep>0$ there are $C, A>0$ such that $
\lr{\eta}_{\!\ga}^{\delta |\be|}\leq CA^{|\be|}|\be|!^{s\delta/(1-\delta)}e^{\ep\lr{\eta}_{\!\ga}^{(1-\delta)/s}}$
then repeating the same arguments estimating \eqref{eq:kurikaesi} we obtain the following  
\begin{lem}
\label{lem:op:a:zou}We have $
{\mathcal Op}(a)( G^{s/(1-\delta)}(\R^n))\subset  G^{s/(1-\delta)}(\R^n)$ if $a(x,\xi, y)\in {\mathcal A_{\delta}^{(s)}}(e^{c\xiga^{\kappa}})$ and $1-\delta>\kappa s$. 
\end{lem}
Let $a_i(x,  \xi, y)\in {\mathcal A^{(s)}_{\delta}}(e^{c_i\xiga^{\kappa}})$ with $1-\delta>\kappa s$ and consider ${\mathcal Op}({ a_1}){\mathcal Op}({ a_2})$. Assume that 
\begin{gather*}
(2\pi)^{-n}\int e^{i(x-y)\eta+i(y-z)\zeta}\chi_{\epsilon_1}(x-y)\chi_{\ep_1}(\eta)\chi_{\ep_2}(y-z)\chi_{\ep_2}(\zeta)\\
\times {a_1}(x, \eta, y){ a_2}(y, \zeta, z)u(z)dyd\eta dz d\eta\\
=(2\pi)^{-n}\int e^{i(x-z)\xi}b_{\ep}((x+z)/2, \xi)u(z)dz d\xi,\quad \ep=(\ep_1,\ep_2)
\end{gather*}
for any $u(z)\in G_0^s(\R^n)$ then $\int e^{i(x-z)\xi}b_{\ep}((x+z)/2,\xi)d\xi$ is equal to
\begin{gather*}
(2\pi)^{-n}\int e^{i(x-y)\eta+i(y-z)\zeta}\chi_{\epsilon_1}(x-y)\chi_{\ep_1}(\eta)\chi_{\ep_2}(y-z)\chi_{\ep_2}(\zeta)\\
\times { a_1}(x, \eta, y){ a_2}(y, \zeta, z)dyd\eta  d\zeta.
\end{gather*}
Write $x+z\to 2 x$, $z-x\to 2z$ and make the change of variables $\eta+\zeta\to 2{ \eta}$, $\eta-\zeta\to 2{\zeta}$ and note that $\int e^{-2i{ z}\xi}b_{\ep}({ x}, \xi)d\xi=({\mathcal F}b_{\ep})(2{ z})$, where ${\mathcal F}b_{\ep}$ is the Fourier transform of $b_{\ep}$. Then the Fourier inversion formula gives
\begin{gather*}
b_{\ep}({ x}, \xi)=\pi^{-2n}\int e^{2i({ x}-{ y}){\zeta}-2i{ z}({ \eta}-\xi)}\chi_{\epsilon_1}({ x}-{ z}-{ y})\chi_{\ep_1}(\eta+\zeta)\chi_{\ep_2}({ y}-{ x}-{ z}) \\
\times \chi_{\ep_2}({\eta}-{\zeta})
{  a_1}( x-z,  \eta+ \zeta,  y)
{ a_2}({ y}, {\eta}-{\zeta}, { x}+{ z}) d{y}d{\zeta}
 d{\eta}d{ z}.
\end{gather*}
After the change of variables $\eta\to(\eta+\zeta)/2+\xi$, $y\to y+z+x$, $z\to z-y$, $\zeta\to (\eta-\zeta)/2$ we let $\ep_i\to 0$ to obtain $b(x,\xi)=\lim_{\ep_i\to 0}b_{\ep}(x,\xi)$. 
Denoting $b((x+{\tilde x})/2, \xi)={\tilde b}(x, \xi, {\tilde x})$ it is clear that 
 ${\tilde b}(x, \xi, {\tilde x})$ is given by
 \begin{equation}
 \label{eq:calop:def}
 \begin{split}
 \pi^{-2n}\int e^{-2i(z\eta-y\zeta)}
{a_1}((x+{\tilde x})/2+y-z, \xi+\eta, (x+{\tilde x})/2+y+z)\\
\times
{ a_2}((x+{\tilde x})/2+y+z, \xi+\zeta, (x+{\tilde x})/2-y+z)dy d\zeta d\eta dz.
\end{split}
 \end{equation}
 \begin{definition}
\label{dfn:calop:seki}\rm Let $a_i(x,  \xi, y)\in {\mathcal A^{(s)}_{\delta}}(e^{c_i\xiga^{\kappa}})$ with $1-\delta>\kappa s$. We define $(a_1\#a_2)(x,\xi,\tilde x)$ by the oscillatory integral \eqref{eq:calop:def}.
\end{definition}
In what follows we write $X=(x, \xi)$, $Y=(y,\eta)$, $Z=(z,\zeta)$ and $\sigma(Y, Z)=\eta z -y\zeta=\lr{\sigma Y, Z}$ where
\[
\sigma=\begin{bmatrix}O&I\\
-I&O
\end{bmatrix} .
\]
\begin{pro}
\label{pro:b:class}Let ${a_i}(x,\xi, y)\in {\mathcal A_{\delta}^{(s)}}(e^{c_i\xiga^{\kappa_i}})$ with $1-\delta>\kappa_i s$. There exist $b_i>0$ such that $a_1\#a_2\in {\mathcal A_{\delta}^{(s)}}(e^{\Sigma b_ic_i\xiga^{\kappa_i}})$ and ${\mathcal Op}(a_1\#a_2)= {\mathcal Op}({a_1})
 {\mathcal Op}({a_2})$.
\end{pro}
\begin{proof}It remains to show $(a_1\#a_2)(x, \xi, {\tilde x})\in {\mathcal A_{\delta}^{(s)}}(e^{\Sigma b_ic_i\xiga^{\kappa_i}})$. Denote 
\[
F(X,Y, Z)={  a_1}(x+y-z, \xi+\eta, x+y+z)
{ a_2}(x+y+z, \xi+\zeta, x-y+z)
\]
and let $\chi_0(x)\in G^{s}(\R)$ be $1$ in $|x|\leq 1/5$ and $0$ for $|x|\geq 1/4$ and denote $\chi(\cdot)=\chi_0(\cdot \xiga^{-1})$ and $\bar\chi(\eta,\zeta)=\chi(\eta)\chi(\zeta)$ and $f^c=1-f$ in general. 
Write
\begin{gather*}
\int e^{-2i\sigma(Y,Z)}F(X, Y, Z){\bar \chi}  dYdZ
+\int e^{-2i\sigma(Y,Z)}F(X, Y, Z){\bar \chi^c}dYdZ=I+II.
\end{gather*}
Since $|\dif_{\xi, \eta,\zeta}^{\al}({\bar\chi}, {\bar\chi^c})|\leq A^{|\al|}|\al|!^s\xiga^{-|\al|}$ and $\lr{\xi+\eta}_{\!\ga}\approx \xiga$, $\lr{\xi+\zeta}_{\!\ga}\approx \xiga$ if ${\bar\chi}\neq 0$, the estimate for $I$ follows from integration by parts. Write 
\begin{equation}
\label{eq:chi:c:bunkai}
\begin{split}
{\bar\chi^c}=\chi^c(\eta)\chi^c(\zeta)
+\chi^c(\eta)\chi(\zeta)
+\chi^c(\zeta)\chi(\eta)
={\varphi}_1+{\varphi}_2+{\varphi}_3
\end{split}
\end{equation}
and let $\chi_1(t)\in G^s(\R)$ be $1$ in $|t|<1$ and $0$ for $|t|\geq 2$. Study  
\begin{equation}
\label{eq:chi:star}
\begin{split}
\int e^{-2i\sigma(Y,Z)}\lr{\eta}^{-2N_2}\lr{\zeta}^{-2N_1}\lr{D_z}^{2N_2}\lr{D_y}^{2N_1}\\
\times \lr{y}^{-2\ell}\lr{z}^{-2\ell}\lr{D_{\zeta}}^{2\ell}\lr{D_{\eta}}^{2\ell}(\dif_x^{\be}\dif_{\xi}^{\al}F{\varphi}_1)(\chi_*+\chi_*^c)dYdZ
\end{split}
\end{equation}
with $\chi_*=\chi_1(\lr{\zeta}\lr{\eta}^{-1})$ and estimate the integrand
\[
\big|\lr{\eta}^{-2N_2}\lr{\zeta}^{-2N_1}\lr{D_z}^{2N_2}\lr{D_y}^{2N_1}\lr{y}^{-2\ell}\lr{z}^{-2\ell}\lr{D_{\zeta}}^{2\ell}\lr{D_{\eta}}^{2\ell}(\dif_x^{\be}\dif_{\xi}^{\al}F{\varphi}_1)\chi_*\big|.
\]
Choosing $N_1=\ell, N_2=N$ and noting $\xiga\leq C\lr{\eta}$, $\lr{\xi+\eta}_{\!\ga}\leq C\lr{\eta}$, $ \lr{\xi+\zeta}_{\!\ga} \leq C\lr{\eta}$ if $\varphi_1\chi_*\neq 0$, it is not difficult to see that this is bounded by
\begin{equation}
\label{eq:copi:b}
\begin{split}
CA^{2N+|\al+\be|}\lr{\eta}^{-2N}\lr{\zeta}^{-2\ell}\lr{y}^{-2\ell}\lr{z}^{-2\ell}\lr{\eta}^{2\delta\ell}(2N)!^s|\al+\be|!^s\\
\times (N^{\delta s/(1-\delta)}+\lr{\eta}^{\delta})^{2N}
(|\be|^{s\delta/(1-\delta)}+\lr{\eta}^{\delta})^{|\be|}e^{c\lr{\eta}^{\kappa}},\;\; \kappa=\max\{\kappa_1,\kappa_2\}.
\end{split}
\end{equation}
From the same arguments estimating \eqref{eq:kurikaesi}, 
one can estimate \eqref{eq:copi:b} by
\[
C_{\ell}A_{\ell}^{|\al+\be|}\lr{\zeta}^{-2\ell}\lr{y}^{-2\ell}\lr{z}^{-2\ell}|\al+\be|!^s|\be|^{s\delta |\be|/(1-\delta)}e^{-c\lr{\eta}^{(1-\delta)/s}}
\]
which proves $\int e^{-2i\sigma(Y,Z)}F(X,Y,Z)\varphi_1\chi_*dYdZ\in {\mathcal A_{\delta}^{(s)}}(e^{-c'\xiga^{\kappa}})$ with $c'>0$. 
For the case $\chi^c_*$, choosing $N_1=N, N_2=\ell$, a repetition of the same argument shows   
$\int e^{-2i\sigma(Y,Z)}F(X,Y,Z)\varphi_1dYdZ\in {\mathcal A_{\delta}^{(s)}}(e^{-c'\xiga^{\kappa}})$.
Next, consider 
\[
|\lr{\eta}^{-2N}\lr{\zeta}^{-2\ell}\lr{D_z}^{2N}\lr{D_y}^{2\ell}
 \lr{y}^{-2\ell}\lr{z}^{-2\ell}\lr{D_{\zeta}}^{2\ell}\lr{D_{\eta}}^{2\ell}\dif_x^{\be}\dif_{\xi}^{\al}F{\varphi}_2|
\]
Since $\lr{\xi+\eta}_{\!\ga}\leq C\lr{\eta}$, $ \lr{\xi+\zeta}_{\!\ga}\approx \xiga \leq C\lr{\eta}$ if $\varphi_2\neq 0$, this is bounded by \eqref{eq:copi:b}, hence the same as $\varphi \chi_*$. The case $\varphi_3$ is similar to the case $\varphi_1\chi_*^c$.
\end{proof}
%

\subsection{Pseudodifferential operators with symbol of type  $\exp{S^{\kappa}_{\rho,\delta}}$}
\label{sec:sindo}

%
\begin{definition}
\label{dfn:gmarus}\rm Let $m=m(x, \xi; \ga)>0$ be a positive function. We define $S^{(s)}_{\!\rho,\delta}(m)$ to be the set of all $a(x,\xi; \ga)\in C^{\infty}(\R^{n}\times \R^{n})$ such that we have
\begin{equation}
\label{eq:marus}
|\dif_x^{\be}\dif_{\xi}^{\al}a(x,\xi; \ga)|\leq CA^{|\al+\be|}|\al+\be|!^s m(x, \xi, \ga)\xiga^{\delta|\be|-\rho|\al|}
\end{equation}
for all $\al$, $\be\in\N^{n}$ with some $C, A>0$ independent of $\ga\geq 1$ and $S_{\!\rho,\delta}(m)$ to be the set of all $a(x, \xi, \ga)$ satisfying \eqref{eq:marus} with $C_{\al\be}$ instead of $CA^{|\al+\be|}|\al+\be|!^s$ which may depend on $\al, \be$ but not on $\ga$. We often write just $a(x,\xi)$ or $m(x, \xi)$ dropping $\ga$. Note that $S_{0,0}^{(s)}(m)={\mathcal { A}^{(s)}_{0}}(m)$.
\end{definition}
\begin{definition}
\label{df:Wyle}\rm For $a(x,\xi)\in S^{(s)}_{\!\rho,\delta}(m)$ we define ${\rm op}(a)$ by
\[
{\rm op}(a)u(x)={\mathcal Op}({\tilde a})u(x), \;\; {\tilde a}(x, \xi, y)=a((x+y)/2, \xi)\in  {\mathcal { A}^{(s)}_{\delta}}(m)
\]
which is called the Wyle quantization of $a$.
\end{definition}
The next lemma is a special case of Proposition \ref{pro:b:class}.
\begin{lem}
\label{lem:S:no:baai}
Let $a_i(x,\xi)\in  {\mathcal A_{\delta}^{(s)}}(e^{c_i\xiga^{\kappa}})$ with $1-\delta>\kappa s$. If we set
\[
b(X)=\pi^{-2n}\int e^{-2i\sigma(Y,Z)} a_1(X+Y)
a_2(X+Z)dY dZ =(a_1\#a_2)(X)
\]
then $b(X)\in {\mathcal A_{\delta}^{(s)}}(e^{c'\xiga^{\kappa}})$ $(c'>0)$ and verifies ${\rm op}(a_1){\rm op}(a_2)={\rm op}(b)$. 
\end{lem}
Let $a_i(x, \xi)\in {\mathcal A_{\delta}^{(s)}}(e^{c_i\xiga^{\kappa}})$ with $1-\delta>\kappa s$. Then
$a_1\#a_2\#a_3$ is given by
\begin{gather*}
\pi^{-4n}\int e^{-2i\sigma(Y, Z)-2i\sigma(S, T)}
 a_1(X+Y)a_2(X+Z+S)a_3(X+Z+T)dYdZdSdT.
\end{gather*}
By the change of variables $Z\to Z-T$, $S\to S+T+Y$, the integral is
\begin{gather*}
\pi^{-4n}\int e^{-2i\sigma(Y, Z)-2i\sigma(S, T)}
 a_1(X+Y)a_2(X+Y+Z+S)a_3(X+Z)dYdZdSdT.
\end{gather*}
Noting that $\int e^{-2i\sigma(S, T)}dT=\pi^{2n}\delta(S)$, this is equal to
\[
\pi^{-2n}\int e^{-2i\sigma(Y, Z)}
 a_1(X+Y)a_2(X+Y+Z)a_3(X+Z)dYdZ.
\]
\begin{lem}
\label{lem:bb1}
Let $m=m(x,\xi; \ga)>0$ be a positive function and $f\in S^{(s)}_{\rho,\delta}(m)$. With $\omega^{\al}_{\be}=e^{-f}\dif_x^{\be}\dif_{\xi}^{\al}e^{f}$ there exist $A, C>0$ such that the following holds.
 \[
\big|\dif_x^{\nu}\dif_{\xi}^{\mu}\omega^{\al}_{\be}\big|\leq CA^{|\nu+\mu+\al+\be|}\xiga^{\delta|\be+\nu|-\rho|\al+\mu|}\sum_{j=0}^{|\al+\be|}m^{|\al+\be|-j}(|\mu+\nu|+j)!^s.
\]
\end{lem}
\begin{cor}
\label{cor:bb}
There are $A, C>0$ such that
\begin{eqnarray*}
|\dif_x^{\be}\dif_{\xi}^{\al}e^{f}|\leq C|e^{f}|A^{|\al+\be|}\xiga^{\delta|\be|-\rho|\al|}(m+|\al+\be|^s)^{|\al+\be|},\quad \al, \be\in\N^n.
\end{eqnarray*}
In particular $e^{f(x, \xi)}\in S^{(s)}_{\!\rho,\delta}(|e^f|e^{s m^{1/s}})\subset {\mathcal A}_{\delta}^{(s)}(e^{|f|+sm^{1/s}})$.
\end{cor}
\begin{proof}The first assertion follows from Lemma \ref{lem:bb1} immediately. Since $m^{N}\leq N!^se^{sm^{1/s}}$ ($s>0$) one can find $C>0$ independent of $s> 1$ such that
\begin{gather*}
\sum_{j=0}^{|\al+\be|}m^{|\al+\be|-j}j!^s\leq e^{sm^{1/s}}\sum_{j=0}^{|\al+\be|}(|\al+\be|-j)!^sj!^s
\leq Ce^{sm^{1/s}}|\al+\be|!^s
\end{gather*}
which proves the second assertion.
\end{proof}
%

\subsection{Composition formula }

Let $\phi(x,\xi)\in S^{(s)}_{\!\rho,\delta}(\xiga^{{\kappa}})$ and from now on we always assume
\begin{equation}
\label{eq:s:con}
0\leq\delta<\rho\leq 1,\quad \rho-\delta>\kappa\geq 0, \quad s> 1.
\end{equation}
Assume in addition that the following condition holds. 
\[
 {\bar\ep}:=(\rho-\delta)/s-\kappa-\max{\{\delta, 1-\rho\}}(s-1)/s>0
\]
which follows from \eqref{eq:s:con} provided that $s$ is chosen sufficiently close to 1. 
Denote the metric defining the class $S_{\!\rho,\delta}$ by $g$;
\[
g_X(Y)=\xiga^{2\delta}|y|^2+\xiga^{-2\rho}|\eta|^2,\quad X=(x,\xi), \;Y=(y,\eta)\in\R^n.
\]
\begin{definition}[\cite{Ni:JPDO}]
\label{dfn:admissible}\rm
A positive function $m(x,\xi; \ga)$ is called $S_{\!\rho,\delta}$ admissible weight if there are positive constants
$C, N$ such that
\[
m(X)\leq Cm(Y)\big(1+\max{\{g_{X}(X-Y), g_Y(X-Y)\}}\big)^N,\; X, Y\in \R^{2n}.
\]
\end{definition}
Note that $S_{\rho,\delta}$ admissible weight is $\sigma$, $g$ temperate weight in \cite{Ho:book3}.
\begin{thm}
\label{thm:matome}Let $p(x,\xi)\in S^{(s)}_{\!\rho,\delta}(w)$ and $w$ be $S_{\!\rho,\delta}$ admissible weight. Then there exists $c>0$ such that for any $l, m\in\N$, one can write $e^{\phi}\#p\#e^{-\phi}$ as
\begin{align*}
\Big(\frac{1}{2}\sum_{k=0}^m\frac{(-1)^k}{k!}{\sum}'\frac{1}{(2i)^{|\al|}\al^0!\al^1!\cdots\al^k!}+\frac{1}{2}\sum_{k=0}^m\frac{1}{k!}{ \sum}'\frac{(-1)^{|\al|}}{(2i)^{|\al|}\al^0!\al^1!\cdots\al^k!}\Big)\\
\times (\sigma \dif_Y)^{\al} (\dif_X^{\al^0}p(X+2Y)\dif_X^{\al^1}\phi(X+Y)
\cdots\dif_X^{\al^k}\phi(X+Y))\big|_{Y=0}\\
+S^{(s)}_{\!\rho,\delta}(w \xiga^{-l(\rho-\delta)})+S^{(s)}_{\!\rho,\delta}(w\xiga^{-{\bar\ep}(m+1)})
+S^{(s/(1-\delta))}_{0, 0}(e^{-c\xiga^{(1-\delta)/s}})
\end{align*}
where $\sum'$ denotes the sum over all $\al^0+\al^1+\cdots+\al^k=\al$, $|\al^0|\leq l-1$, $1\leq |\al^j|\leq l$, $1\leq j\leq k$ and $\sigma \dif_Y=(\dif_{\eta}, -\dif_{y})$. 
\end{thm}
Note that the main terms with $|\al|-k$ odd cancel out. Denote by $\sum''$ the sum over all $\al^1+\cdots+\al^k=\al$, $1\leq |\al^j|\leq l$ ($1\leq j\leq k$), then
\begin{pro}
\label{thm:matome:b}Let $a(x,\xi)\in S^{(s)}_{\!\rho,\delta}(w)$ and $w$ be $S_{\!\rho,\delta}$ admissible weight. There exists $c>0$ such that for any $N\in\N$ there are $l, m\in\N$ so that $(ae^{\phi})\#e^{-\phi}$ is given by
\begin{align*}
\sum_{k=0}^m\frac{(-1)^k}{k!}{\sum}''\frac{1}{(2i)^{|\al|}\al^1!\cdots\al^k!}
(\sigma \dif_X)^{\al}
(a(X)\dif_X^{\al^1}\phi(X)
\cdots\dif_X^{\al^k}\phi(X))\\
+S^{(s)}_{\!\rho,\delta}(w \xiga^{-N})
+S^{(s/(1-\delta))}_{0, 0}(e^{-c\xiga^{(1-\delta)/s}})
\end{align*}
and $e^{\phi}\#(ae^{-\phi})$ is given by
\begin{gather*}
\sum_{k=0}^m\frac{1}{k!}{\sum}''\frac{(-1)^{|\al|}}{(2i)^{|\al|}\al^1!\cdots\al^k!}
(\sigma \dif_X)^{\al}
(a(X)\dif_X^{\al^1}\phi(X)
\cdots\dif_X^{\al^k}\phi(X))\\
+S^{(s)}_{\!\rho,\delta}(w \xiga^{-N})
+S^{(s/(1-\delta))}_{0, 0}(e^{-c\xiga^{(1-\delta)/s}}).
\end{gather*}
\end{pro}
\begin{cor}
\label{cor:matome}For any $N\in\N$ one can find $l, m\in\N$ such that
\begin{gather*}
e^{\phi}\#e^{-\phi}=1+\sum_{k=2}^m\frac{1}{k!}{\sum}''\frac{(-1)^{|\al|}}{(2i)^{|\al|}\al^1!\cdots\al^k!}(\sigma \dif_X)^{\al}\\
\times (\dif_X^{\al^1}\phi(X)
\cdots\dif_X^{\al^k}\phi(X))
+S_{\!\rho,\delta}(w\xiga^{-N}).
\end{gather*}
\end{cor}
\begin{proof} It suffices to note that 
$\sum_{|\al|=l}(\sigma \dif_X)^{\al}(\dif_X^{\al}\phi)/\al!=0$ for $l\geq 1$.
\end{proof}
\begin{thm}
\label{thm:matome:a:1:2}Let $a_i(x,\xi)\in S^{(s)}_{\!\rho,\delta}(w_i)$ and $w_i$ be $S_{\!\rho,\delta}$ admissible weights. There is $c>0$ such that for any $l\in\N$ we have
\begin{align*}
a_1\#a_2=\sum_{|\al|\leq l-1}\frac{(-1)^{|\al|}}{(2i)^{|\al|}\al!}\{(\sigma \dif_X)^{\al}a_1(X)\}\dif_X^{\al}a_2(X)\\
+S^{(s)}_{\!\rho,\delta}(w_1w_2 \xiga^{-l(\rho-\delta)})
+S^{(s/(1-\delta))}_{0, 0}(e^{-c\xiga^{(1-\delta)/s}}).
\end{align*}
\end{thm}
\begin{pro}
\label{pro:0:0}
Let $a \in S_{\rho,\delta}^{(s)}(\xiga^m)$ and $\phi\leq -\bar c\xiga^{\kappa}$ on the support of $a$ with $\bar c>0$ and $(\rho-\delta)/(2s-1)>\kappa$. Then there are $c_i>0$ such that
\begin{align*}
a\#e^{\phi}=S^{(s)}_{\rho,\delta}(\xiga^me^{-\bar c\xiga^{\kappa}})
+ S^{(s/(1-\delta))}_{0, 0}(  e^{-c_1\xiga^{(1-\delta)/s}})\subset {\mathcal A}_{\delta}^{(s/(1-\delta))}(e^{-c_2\xiga^{\kappa}})
\end{align*}
and the same for $e^{\phi}\#a$.

\end{pro}
%

\subsection{Proof of Theorem \ref{thm:matome}}

Let $p\in S^{(s)}_{\!\rho,\delta}(w)$ where $w$ is $S_{\!\rho,\delta}$ admissible weight and consider  
\begin{align*}
e^{\phi}\# p\#e^{-\phi}=\pi^{-2n}\int e^{-2i\sigma(Y,Z)}p(X+ Y+Z)e^{\phi(X+Y)-\phi(X+Z)}dYdZ.
\end{align*}
Let $\bar \chi(\eta,\zeta)$ be as before. Denote $
{\munderbar \chi}(y, z)=\chi_0(\lr{y})\chi_0(\lr{z})$ and ${\munderbar \chi^c}=1-{\munderbar \chi}$.
Write, disregarding the factor $\pi^{-2n}$, and denote
\begin{equation}
\label{eq:bunkai}
\int e^{-2i\sigma(Y,Z)}p(X+ Y+Z) e^{\phi(X+Y)-\phi(X+Z)}\big({\munderbar \chi}{\bar \chi}+{\munderbar \chi^c}{\bar \chi}+{\bar \chi^c}\big) dYdZ
\end{equation}
by $I_1+I_2+I_3$. After the change of variables $Z\to Z+Y$, $I_1$ turns to be 
\begin{equation}
\label{eq:sitamati}
\int e^{-2i\sigma(Y,Z)}p(X+ 2Y+Z) e^{\phi(X+Y)-\phi(X+Y+Z)}{\varphi}(X,Y,Z)dYdZ
\end{equation}
where ${\varphi}(X, Y, Z)={ \munderbar\chi}(y, y+z){\bar\chi}(\eta, \eta+\zeta)$. Noe that 
\begin{equation}
\label{eq:doti}
{\bar\chi}(\eta,\eta+\zeta)\neq 0\Longrightarrow \lr{\xi+\eta}_{\!\ga}\approx \xiga, \; \lr{\xi+\eta+\theta \zeta}_{\!\ga}\approx \xiga,\;|\theta|\leq 1.
\end{equation}
This will be used without mention hereafter. 
To simplify notation, we denote 
\[
\dif_X^{\al}=\dif_x^{\al_x}\dif_{\xi}^{\al_{\xi}},\;\;\al=(\al_{x}, \al_{\xi})\in \N^{2n},\;\; \ep(\al)=\delta|\al_{x}|-\rho|\al_{\xi}|,\;\;\sigma \al=(\al_{\xi}, -\al_{x})
\]
so that $\ep(\al)+\ep(\sigma\al)=-(\rho-\delta)|\al|$. Similarly we write $\dif_{X,Y}^{\al}$, $\al\in\N^{4n}$ or $\dif_{X,Y,Z}^{\al}$, $\al\in\N^{6n}$ and $\ep(\al)$, $\sigma\al$ can be understood analogously.  Then we have
\begin{equation}
\label{eq:chi:hat}
|\dif_{X,Y,Z}^{\al}{\varphi(X,Y, Z)}|\leq CA^{|\al|}|\al|!^s\xiga^{-|\al_{\xi}|}.
\end{equation}
Denote $
\psi(X, Y, Z)=\phi(X+Y)-\phi(X+Y+Z)$ and write $
\dif_{X, Y}^{\al}\psi(X, Y, Z)=Z\cdot\int_0^1\nabla_X\dif_{X, Y}^{\al}\phi(X+Y+\theta Z)d\theta$. By \eqref{eq:doti} we have $\big|\dif_{X, Y}^{\al}\psi(X, Y, Z)\big|\leq CA^{|\al|}|\al|!^s\xiga^{\ep(\al)}\xiga^{\kappa} (|z|\xiga^{\delta}+|\zeta|\xiga^{-\rho})\leq CA^{|\al|}|\al|!^s\xiga^{\ep(\al)}\xiga^{\kappa}g_X^{1/2}(Z)$.
Applying Corollary \ref{cor:bb} with $m=\xiga^{\kappa}g_X^{1/2}(Z)$ we conclude
\begin{equation}
\label{eq:sakura:lem}
\begin{split}
\big|\dif_{X, Y}^{\al}&\psi(X, Y, Z)\big|\leq CA^{|\al|}|\al|!^s\xiga^{\ep(\al)}\xiga^{\kappa}g_X^{1/2}(Z),\\
|\dif_{X,Y}^{\al}e^{\psi(X, Y, Z)}|&\leq CA^{|\al|}
\xiga^{\ep(\al)}\big(\xiga^{\kappa}{ g}_X^{1/2}(Z)+|\al|^s\big)^{|\al|}e^{|\psi(X, Y, Z)|}.
\end{split}
\end{equation}
 By the Taylor formula, we have
\begin{equation}
\label{eq:psi:Tay}
\begin{split}
\psi=-\sum_{1\leq |\al|\leq l}\frac{1}{\al!}\dif_X^{\al}\phi(X+Y)Z^{\al}+\sum_{|\mu|=l+1}{\tilde r}_{l\mu}(X, Y, Z)Z^{\mu},\\
{\tilde r}_{l\mu}(X, Y, Z)=\frac{l+1}{\mu!}\int_0^1(1-\theta)^l\dif_X^{\mu}\phi(X+Y+\theta Z)d\theta
 \end{split}
\end{equation}
where one has 
$\big|\dif_{X, Y, Z}^{\be}{\tilde r}_{l\mu}\big|\leq C_lA_l^{|\be|}|\be|!^s\xiga^{\kappa+\ep(\mu)+\ep(\be)}$. 
Write
\begin{equation}
\label{eq:e:psi}
e^{\psi}=\sum_{k=0}^m\frac{\psi^k}{k!}+\frac{\psi^{m+1}}{m!}\int_0^1(1-\theta)^me^{\theta \psi}d\theta=\sum_{k=0}^m\frac{\psi^k}{k!}+ R_m.
\end{equation}
One can write $
(\xiga^{-\rho}\dif_z)^{\gamma}(\xiga^{\delta}\dif_{\zeta})^{\be}\big(\psi^{m+1}e^{\theta\psi}\big)=e^{\theta\psi}\sum_{j=0}^{m+1}\psi^{m+1-j}q_{j}^{(\be,\gamma)}$
where $q_{j}^{(\be,\gamma)}\in S^{(s)}_{\!\rho,\delta}(\xiga^{-\varepsilon j})$ with $\varepsilon=\rho-\delta-\kappa$ 
then  from \eqref{eq:sakura:lem} we see
\begin{equation}
\label{eq:zenbu:cor}
\begin{split}
|(\xiga^{-\rho}\dif_z)^{\gamma}(\xiga^{\delta}\dif_{\zeta})^{\be}\dif_{X,Y}^{\al}R_m|\leq  C_{\be,\gamma}A_{\be,\gamma}^{|\al|}\xiga^{\ep(\al)}\\
\times \sum_{j=0}^{m+1}\{\xiga^{\kappa}g_X^{1/2}(Z)\}^{m+1-j}\xiga^{-\varepsilon j}\big(\xiga^{\kappa}g_X^{1/2}(Z)+|\al|^s\big)^{|\al|}e^{|\psi|}.
\end{split}
\end{equation}
In view of \eqref{eq:psi:Tay} we can write 
\begin{equation}
\label{eq:psi:k:power:lem}
\begin{split}
\psi^k=(-1)^k\Big(\sum_{1\leq|\al|\leq l}\frac{1}{\al!}\dif_X^{\al}\phi(X+Y)Z^{\al}\Big)^k+r^{\psi}_{lk}(X,Y,Z),\\
r^{\psi}_{lk}=\sum_{l+k\leq |\mu|\leq k(l+1)}r_{lk\mu}(X,Y,Z)Z^{\mu},\quad  r^{\psi}_{l0}=0.
\end{split}
\end{equation}
Since $|\dif_{X, Y}^{\al}(\dif_X^{\mu}\phi(X+Y))|\leq CA^{|\al|}|\al|!^s\xiga^{\kappa+\ep(\mu)+\ep(\al)}$ we see from \eqref{eq:psi:Tay} that
\begin{equation}
\label{eq:psi:k:power:lem:bis}
|\dif_{X, Y, Z}^{\al}r_{lk\mu}|\leq C_{lk}A_{lk}^{|\al|}|\al|!^s\xiga^{k\kappa+\ep(\mu)+\ep(\al)}.
\end{equation}
\begin{lem}
\label{lem:weight}
If $w$ is $S_{\!\rho,\delta}$ admissible weight there are $C, N>0$ such that
\[
w(X+2Y+Z)\leq Cw(X)(1+g_X(Y))^{N}(1+g_X(Z))^{N},\quad {\bar\chi}(\eta,\eta+\zeta)\neq 0.
\]
\end{lem}
\begin{proof}Note that $g_{X+Y}\approx g_X$ and $g_{X+2Y+Z}\approx g_X$ if ${\bar\chi}(\eta,\eta+\zeta)\neq 0$. Hence
\begin{gather*}
w(X+2Y+Z)\leq Cw(X+Y)
 (1
+\max{\{g_{X+Y}(Y+Z), g_{X+2Y+Z}(Y+Z)\}})^{N_1}\\
\leq C_1w(X+Y)(1+g_X(Y+Z))^{N_1}\leq C_2w(X)(1+g_X(Y))^{N_2}
(1+g_X(Y+Z))^{N_1}.
\end{gather*}
Since $g_X(Y+Z)\leq 2(g_X(Y)+g_X(Z))$ the proof is clear.
\end{proof}
Denote $q(X,Y,Z)=p(X+2Y+Z)\varphi$ and write
\begin{align*}
q(X,Y,Z)=&\sum_{|\al|\leq l-1}\frac{1}{\al!}\dif_Z^{\al}q(X, Y, 0)Z^{\al}+\sum_{|\mu|=l}r^q_{l\mu}(X,Y,Z)Z^{\mu},\\
r^q_{l\mu}&=\frac{l}{\mu!}\int_0^1 (1-\theta)^{l-1}(\dif_Z^{\mu} q)(X, Y, \theta Z) d\theta,\quad l\geq 1.
\end{align*}
In view of Lemma \ref{lem:weight}, \eqref{eq:chi:hat} and $\lr{\xi+2\eta+\theta\zeta}\approx\xiga$ 
there is $N$ such that 
\begin{equation}
\label{eq:p:Tay:lem}
|\dif_{X,Y,Z}^{\al}r^q_{l\mu}|\leq CA^{|\al|}|\al|!^s w (X)(1+g_X(Y)+g_X(Z))^{N}\xiga^{\ep(\al)+\ep(\mu)}.
\end{equation}
We consider
\[
\int e^{-2i\sigma(Y,Z)}(\sum _kr^{\psi}_{lk}+R_m)q(X, Y, Z) dYdZ=J_1+J_2
\]
where $R_m=R_m(X, Y, Z)$ and $r^{\psi}_{lk}=r^{\psi}_{lk}(X,Y, Z)$ are given in \eqref{eq:e:psi} and \eqref{eq:psi:k:power:lem}. Introduce the following differential operators and symbols
\[
\left\{\begin{array}{ll}
L=1+4^{-1}\xiga^{2\rho}|D_{\eta}|^2+4^{-1}\xiga^{-2\delta}|D_y|^2=1+g_X^{\sigma}(\sigma D_Y)/4,\\[3pt]
M=1+4^{-1}\xiga^{2\delta}|D_{\zeta}|^2+4^{-1}\xiga^{-2\rho}|D_z|^2=1+g_X(\sigma D_Z)/4,\\[3pt]
\Phi=1+\xiga^{2\rho}|z|^2+\xiga^{-2\delta}|\zeta|^2
=1+g^{\sigma}_X(Z),\\[3pt]
\Psi=1+\xiga^{2\delta}|y|^2+\xiga^{-2\rho}|\eta|^2=1+g_X(Y)\end{array}\right.
\]
so that $\Phi^{-N}L^Ne^{-2i\sigma(Y,Z)}=e^{-2i\sigma(Y,Z)}$ and $\Psi^{-\ell}M^{\ell}e^{-2i\sigma(Y,Z)}=e^{-2i\sigma(Y,Z)}$. Using these relations,  we make integration by parts in $J_2$, which leads us to estimate 
$|\Phi^{-N}L^N\Psi^{-\ell}M^{\ell}\dif_X^{\al}(R_m(X,Y,Z)q(X, Y, Z))|$.  
Here note that
\begin{gather*}
|(\xiga^{-\delta}\dif_y)^{\be}(\xiga^{\rho}\dif_{\eta})^{\al}\Psi^{-\ell}|\leq C_{\ell}A_{\ell}^{|\al+\be|}|\al+\be|!\Psi^{-\ell},\quad \al,\be\in\N^n,\\
|\dif_{X,Y,Z}^{\al}q(X, Y, Z)|\leq CA^{|\al|}|\al|!^sw(X+2Y+Z)\xiga^{\ep(\al)}.
\end{gather*}
Then from \eqref{eq:zenbu:cor} and Lemma \ref{lem:weight}, this is bounded by
\begin{gather*}
C_{\ell}A_{\ell}^{2N+|\al|}\xiga^{\ep(\al)}\Phi^{-N}\Psi^{-\ell}
\sum_{j=0}^{m+1}\{\xiga^{\kappa}g_X^{1/2}(Z)\}^{m+1-j}\xiga^{-\varepsilon j}
\big(\xiga^{\kappa}g_X^{1/2}(Z)\\
+(2N+|\al|)^s\big)^{2N+|\al|}
 w(X)(1+g_X(Y))^{N_1}(1+g_X(Z))^{N_1}e^{c\xiga^{\kappa}g_X^{1/2}(Z)}
\end{gather*}
and hence by
\begin{align*}
CA^{2N+|\al|}\Phi^{-N}\Psi^{-\ell+N_1}\xiga^{\ep(\al)}\sum_{j=0}^{m+1}\{\xiga^{\kappa}g_X^{1/2}(Z)\}^{m+1-j}\xiga^{-\varepsilon j}
 (\xiga^{\kappa}g_X^{1/2}(Z)\\
+N^s)^{2N}(\xiga^{\kappa}g_X^{1/2}(Z)
+|\al|^s)^{|\al|}
 w(X)(1+g_X(Z))^{N_1}e^{c\xiga^{\kappa}g_X^{1/2}(Z)}.
\end{align*}
Writing $
A^{2N}\Phi^{-N}(\xiga^{\kappa}g_X^{1/2}(Z)+N^s)^{2N}
=\big(A\xiga^{\kappa}g_X^{1/2}(Z)/\Phi^{1/2}+AN^s/\Phi^{1/2}\big)^{2N}$ 
we choose the maximal $N=N(Z,\xi)\in\N$ such that $
AN^s\leq {\bar c}\,\Phi^{1/2}$ with ${\bar c}>0$. Noting that $\Phi^{1/2}=\xiga^{\rho-\delta}g_X^{1/2}(Z)$  and $\xiga^{\kappa}g_X^{1/2}(Z)\Phi^{-1/2}=\xiga^{-\varepsilon}$ and choosing ${\bar c}$ small and $\ga\geq \ga_0$ large we have
\[
\Big(\frac{A\xiga^{\kappa}g_X^{1/2}(Z)}{\Phi^{1/2}}+\frac{AN^s}{\Phi^{1/2}}\Big)^{2N}\leq Ce^{-c_1\Phi^{1/2s}}=Ce^{-c_1\xiga^{(\rho-\delta)/s}g_X^{1/2s}(Z)}
\]
  Since $|z|\leq C$ and $|\zeta|\leq C\xiga$ on the support of $\varphi$ one has
\begin{gather*}
\xiga^{\kappa}g_X^{1/2}(Z)=\xiga^{\kappa-(\rho-\delta)/s}g_X^{(s-1)/2s}(Z)\big(\xiga^{(\rho-\delta)/s}g_X^{1/2s}(Z)\big)\\
\leq C\xiga^{\kappa-(\rho-\delta)/s}\xiga^{\max{\{\delta, 1-\rho\}}(s-1)/s}\big(\xiga^{(\rho-\delta)/s}g_X^{1/2s}(Z)\big)
\leq \xiga^{-{\bar\ep}}\Phi^{1/2s}.
\end{gather*}
Noting $
(\Phi^{1/2s})^{|\al|}\leq \ep^{-|\al|}|\al|!e^{\ep\Phi^{1/2s}}$ for any $\ep>0$ 
 it follows that (recall $\varepsilon\geq {\bar\ep}$)
\begin{align*}
\xiga^{-\varepsilon j}\{\xiga^{\kappa}g_X^{1/2}(Z)\}^{m+1-j}(1+g_X(Z))^{N_1}
 (\xiga^{\kappa}g_X^{1/2}(Z)
+|\al|^s)^{|\al|}\\
 \times e^{c\xiga^{\kappa}g_X^{1/2}(Z)}e^{-c_1\,\Phi^{1/2s}}
\leq C_{m,\ell}A_{m,\ell}^{|\al|}|\al|^{s|\al|}\xiga^{-{\bar \ep}(m+1)}e^{-c'\Phi^{1/2s}}.
\end{align*}
Since $\Phi^{l}\leq C_ {\ep}e^{\ep\Phi^{1/2s}}$ for any $\ep>0$, $|\Phi^{-N}L^N\Psi^{-\ell}M^{\ell}\dif_X^{\al}(qR_m)|$ is bounded by 
$CA^{|\al|}|\al|^{s|\al|}\xiga^{\ep(\al)}\xiga^{-{\bar\ep}(m+1)}w(X)\Psi^{-\ell'}\Phi^{-\ell'}$ 
 where $\ell'>n/2$. Recalling $\int \Theta^{-\ell'}\Phi^{-\ell'}dYdZ=C$,  we conclude $
J_2\in S^{(s)}_{\rho,\delta}(\xiga^{-{\bar\ep}(m+1)}w)$. 
Denote $R_{lk\mu}=r_{lk\mu}(X,Y,Z) q(X, Y, Z)$ with $r_{lk\mu}$ given in \eqref{eq:psi:k:power:lem} and consider
\begin{gather*}
\dif_X^{\al}J_1=\sum_{k=1}^m\sum_{l+k\leq |\mu|\leq k(l+1)}\int e^{-2i\sigma(Y,Z)}(\sigma D_Y/2)^{\mu}\dif_X^{\al}R_{lk\mu}(X,Y,Z) dY dZ.
\end{gather*}
Since it follows from \eqref{eq:psi:k:power:lem:bis} and Lemma \ref{lem:weight} that 
$|\dif_{X,Y,Z}^{\al}R_{lk\mu}|$ is bounded by $CA^{|\al|}|\al|!^s\omega(X)(1+g_X(Y))^{N_1}(1+g_X(Z))^{N_1}\xiga^{\ep(\al)+k\kappa+\ep(\mu)}$ 
so that we have
\begin{gather*}
|\Phi^{-\ell}L^{\ell}\Psi^{-\ell}M^{\ell}(\sigma D_Y/2)^{\mu}\dif_X^{\al}R_{lk\mu} |\leq CA^{|\al|} w(X)|\al|!^s\\
\times \xiga^{\ep(\al)+k\kappa+\ep(\mu)+\ep(\sigma\mu)}
\Phi^{-\ell+N_1}\Psi^{-\ell+N_1}.
\end{gather*}
Since $k\kappa+\ep(\mu)+\ep(\sigma\mu)\leq k\kappa-(k+l)(\rho-\delta)\leq \ep k-l(\rho-\delta)$ for $|\mu|\geq k+l$, this is bounded by $
CA^{|\al|}|\al|!^s\xiga^{\ep(\al)-\varepsilon k-l(\rho-\delta)}\Phi^{-\ell+N_1}\Psi^{-\ell+N_1}w(X)$. Choosing $\ell$ such that $\ell-N_1>n/2$, we see $J_1\in S^{(s)}_{\!\rho,\delta}(w \xiga^{-\varepsilon-l(\rho-\delta)})$. 
We next study 
\[
J_3=\int e^{-2i\sigma(Y,Z)}\big(\sum_{1\leq |\al|\leq l} \dif_X^{\al}\phi(X+Y)
Z^{\al}/\al!\big)^k\sum_{|\mu|=l}r^q_{l\mu}Z^{\mu}dYdZ.
\]
 Note that the integrand is a 
 sum of $R^{\psi}_{\be}Z^{\be}q^q_{l\mu}Z^{\mu}$, $k\leq |\be|\leq kl$, $|\mu|=l$  where
\[
|\dif_{X,Y}^{\al}R^{\psi}_{\be}|\leq CA^{|\al|}|\al|!^s\xiga^{k\kappa+\ep(\al)+\ep(\be)}.
\]
Noting  \eqref{eq:p:Tay:lem}, we see that $|\Phi^{-\ell}L^{\ell}\Psi^{-\ell}M^{\ell}(\sigma D_Y/2)^{\mu+\be}\dif_X^{\al}(R^{\psi}_{\be}r^q_{l\mu})|$ is bounded by $
CA^{|\al|}|\al|!^s\xiga^{\ep(\al)-l(\rho-\delta)}\Phi^{-\ell+N_1}\Psi^{-\ell+N_1}w$ from similar arguments for $J_1$, and then $J_3\in S^{(s)}_{\!\rho,\delta}(w \xiga^{-l(\rho-\delta)})$. It remains to consider
\begin{gather*}
\sum_{|\al|\leq l-1}\frac{1}{\al!}\sum_{k=0}^m\frac{(-1)^k}{k!}\int e^{-2i\sigma(Y,Z)}\dif_Z^{\al}q(X, Y, 0)Z^{\al}\\
\times \Big(\sum_{1\leq|\be|\leq l}\frac{1}{\be!}\dif_X^{\be}\phi(X+Y)Z^{\be}\Big)^k dYdZ
\end{gather*}
which is 
\begin{gather*}
\sum_{k=0}^m\frac{(-1)^k}{k!}\int e^{-2i\sigma(Y,Z)}\Big({\sum}'\frac{1}{\al^0!\al^1!\cdots\al^k!}\dif_Z^{\al^0}q(X,Y,0)\dif_X^{\al^1}\phi(X+Y)\\
\cdots\dif_X^{\al^k}\phi(X+Y)\Big)Z^{\al}dYdZ.
\end{gather*}
Recalling $Z^{\al}e^{-2i\sigma(Y,Z)}=(-\sigma D_Y/2)^{\al}e^{-2i\sigma(Y,Z)}$ and $
\int e^{-2i\sigma(Y, Z)}dZ=\pi^{2n}\delta(Y)$ 
and noting $\varphi(X,0,0)=1$, $\dif_{Y, Z}^{\al}\varphi(X, Y, Z)|_{Y=Z=0}=0$ for $|\al|\geq 1$
it yields
\begin{gather*}
\sum_{k=0}^m\frac{(-1)^k}{k!}(\sigma D_Y/2)^{\al}\big({\sum}'\frac{1}{\al^0!\al^1!\cdots\al^k!}\dif_X^{\al^0}p(X+2Y)\\
\times \dif_X^{\al^1}\phi(X+Y)
\cdots\dif_X^{\al^k}\phi(X+Y)\big)\big|_{Y=0}=I(m,l).
\end{gather*}
Thus  $I_1=I(m,l)+S^{(s)}_{\!\rho,\delta}(w \xiga^{-l(\rho-
\delta)})+S^{(s)}_{\!\rho,\delta}(w \xiga^{-{\bar\ep}(m+1)})$. 
Turn to $I_2$. After integration by parts, we have
\[
\int e^{-2i\sigma(Y,Z)}(|y|^2+|z|^2)^{-N}(|D_{\zeta}|^2+|D_{\eta}|^2)^N FdYdZ
\]
 where $F=p(X+Y+Z)e^{\psi}{\tilde\varphi}(X,Y,Z)$ and ${\tilde\varphi}={\munderbar\chi^c}(y,z){\bar\chi}(\eta,\zeta)$. Since $|\psi|\leq C\xiga^{\kappa}$, applying Corollary \ref{cor:bb} it is not difficult to see
\begin{align*}
\big|(|D_{\zeta}|^2+|D_{\eta}|^2)^N(|y|^2+|z|^2)^{-N}\dif_X^{\al}F\big|
\leq CA^{2N+|\al|} w(X+Y+Z)\xiga^{\ep(\al)}\\
\times (\xiga^{\kappa}
+|\al|^s)^{|\al|}\xiga^{-2\rho N} (\xiga^{\kappa}+N^s)^{2N}(|y|^2+|z|^2)^{-N}e^{c\xiga^{\kappa}}.
\end{align*}
Choose $N\in\N$ such that 
$A^{2N}\xiga^{-2\rho N}\big( \xiga^{\kappa}+N^s\big)^{2N}\leq Ce^{-c\xiga^{\rho/s}}$ (recall $\rho>{\kappa}$).
Since $g_X\approx g_{X+Y+Z}$ if ${\bar\chi}(\eta, \zeta)\neq 0$, the proof of Lemma \ref{lem:weight} shows $w (X+Y+Z)\leq C w (X)(1+g_X(Y)+g_X(Z))^{N_2}$.  Note that
\begin{gather*}
g_X(Y)+g_X(Z)\leq \xiga^{2\delta}(|y|^2+|z|^2)+2\xiga^{2(1-\rho)}\\
\leq C\xiga^{2\max{\{\delta, 1-\rho}\}}(1+|y|^2+|z|^2)
\end{gather*}
and for any $\ep>0$ there are $C, A>0$ such that $
\xiga^{\kappa|\al|}\leq CA^{|\al|}|\al|^{s|\al|}e^{\ep \xiga^{\rho/s}}$. Since $\xiga^{-2(n+1)}\int (|y|^2+|z|^2)^{-n-1}{\tilde\varphi }dYdZ\leq C$ we obtain $I_2\in S^{(s)}_{\!\rho,\delta}(e^{-c\xiga^{\rho/s}})$ 
hence clearly belongs to $S^{(s)}_{\!\rho,\delta}(w \xiga^{-{\bar\ep}(m+1)})$. 
To estimate $I_3$, it suffices to repeat the same arguments estimating \eqref{eq:chi:star}. Writing ${\bar\chi^c}(\eta,\zeta)$ as \eqref{eq:chi:c:bunkai} study 
\[
\big|\lr{\eta}^{-2N}\lr{\zeta}^{-2\ell}\lr{D_z}^{2N}\lr{D_y}^{2\ell}\lr{y}^{-2\ell}\lr{z}^{-2\ell}\lr{D_{\zeta}}^{2\ell}\lr{D_{\eta}}^{2\ell}(\dif_X^{\al}F{\varphi}_1)\chi_*\big|
\]
where $F=p(X+Y+Z)e^{\psi}$ and $\psi=\phi(X+Y)-\phi(X+Z)$. 
Noting that $\lr{\xi+\eta+\zeta}_{\!\ga}\leq C\lr{\eta}$ and $\lr{\xi+\eta}_{\!\ga}\leq C\lr{\eta}$, $ \lr{\xi+\zeta}_M \leq C\lr{\eta}$ if $\varphi_1\chi_*\neq 0$ it is not difficult to see that  
this is bounded by
\begin{equation}
\label{eq:copi}
\begin{split}
CA^{2N+|\al|}\lr{\eta}^{-2N}\lr{\zeta}^{-2\ell}\lr{y}^{-2\ell}\lr{z}^{-2\ell}w (X+Y+Z)\lr{\eta}^{2\delta\ell}\lr{\eta}^{6\ell\kappa}\\
\times (\lr{\eta}^{\kappa}+N^s)^{2N}\lr{\eta}^{2\delta N}
(\lr{\eta}^{\kappa}+|\al|^s)^{|\al|}\lr{\eta}^{\delta|\al|}
e^{|\psi|}.
\end{split}
\end{equation}
Writing $A^{2N}\lr{\eta}^{-2N+2\delta N}\big(\lr{\eta}^{\kappa}+N^s\big)^{2N}
=\big( A\lr{\eta}^{\kappa}/\lr{\eta}^{1-\delta}+AN^s/\lr{\eta}^{1-\delta}\big)^{2N}$ 
we choose $N$ so that the right-hand side is bounded by $Ce^{-c\lr{\eta}^{(1-\delta)/s}}$ with some $c>0$. Since $g_Y(Y+Z)\leq C(\lr{\eta}^{2\delta}(|y|^2+|z|^2)+|\eta|^2)\leq C\lr{\eta}^2\lr{y}^2\lr{z}^2$ and $\xiga$, $\lr{\xi+\eta+\zeta}_{\!\ga}\leq C\lr{\eta}$ if $\varphi_1\chi_*\neq 0$, it follows that
\begin{gather*}
w (X+Y+Z)\leq Cw (X)(1+\max{\{g_X(Y+Z), g_{X+Y+Z}(Y+Z)\}})^{N_1}\\
\leq Cw (X)\lr{\eta}^{2N_1}\lr{y}^{2N_1}\lr{z}^{2N_1}.
\end{gather*}
It is clear that for any $\ep>0$ there are $A, C>0$ such that (recall $\kappa+\delta<\rho\leq 1$) $\lr{\eta}^{(\kappa+\delta)|\al|}$ and $|\al|^{s|\al|}\lr{\eta}^{\delta|\al|}$ are bounded by $CA^{|\al|}|\al|^{s|\al|/(1-\delta)}e^{\ep \lr{\eta}^{(1-\delta)/s}}$. 
Since $1-\delta\geq \rho-\delta>{\kappa} s$ one sees that \eqref{eq:copi} is estimated by
\begin{align*}
CA^{|\al|}\lr{\zeta}^{-2\ell}\lr{y}^{-2\ell+2N_1}\lr{z}^{-2\ell+2N_1}w (X)|\al|^{s|\al|/(1-\delta)}e^{-c\lr{\eta}^{(1-\delta)/s}}.
\end{align*}
Similarly, if $\chi^c_*$ is concerned, it suffices to exchange the role of $\eta$ and $\zeta$.  
Since $e^{-c\lr{\eta}^{(1-\delta)/s}}\leq e^{-c_1\xiga^{(1-\delta)/s}}e^{-c_2\lr{\eta}^{(1-\delta)/s}}$ with some $c_i>0$ we conclude
\[
\int e^{-2i\sigma(Y,Z)}p(X+Y+Z)e^{\psi(X,Y,Z)}\varphi_1dYdZ\in S^{(s/(1-\delta))}_{0, 0}(e^{-c\xiga^{(1-\delta)/s}}).
\]
Turn to $\int e^{-2i\sigma(Y,Z)}\dif_X^{\al}(F\varphi_2)dYdZ$. Consider
\begin{align*}
|\lr{\eta}^{-2N}\lr{\zeta}^{-2\ell}\lr{D_z}^{2N}\lr{D_y}^{2\ell}
 \lr{y}^{-2\ell}\lr{z}^{-2\ell}\lr{D_{\zeta}}^{2\ell}\lr{D_{\eta}}^{2\ell}\dif_X^{\al}F{\varphi}_2|.
\end{align*}
Since $\lr{\xi+\eta+\zeta}_{\!\ga}\leq C\lr{\eta}$ and $\lr{\xi+\eta}_{\!\ga}\leq C\lr{\eta}$, $ \lr{\xi+\zeta}_{\!\ga}\approx \xiga \leq C\lr{\eta}$ if $\varphi_2\neq 0$,  this is bounded by \eqref{eq:copi}. The rest of the argument is the same as the case $\varphi \chi_*$. The case $\varphi_3$ is similar to the case $\varphi_1\chi_*^c$. Thus there is $c>0$ such that $I_3\in S^{(s/(1-\delta))}_{0, 0}(w  e^{-c\xiga^{(1-\delta)/s}})$.

By interchanging the roles of $Y$ and $Z$ and repeat the same argument, we obtain $\tilde I_j$ instead of $I_j$ where $\tilde I_j$ belongs to the same class as $I_j$ for $j=2,3$ respectively and $\tilde I_1=\tilde I(m,l)+S^{(s)}_{\!\rho,\delta}(w \xiga^{-l(\rho-
\delta)})+S^{(s)}_{\!\rho,\delta}(w\xiga^{-{\bar\ep}(m+1)})$ where
\begin{gather*}
\tilde I(m,l)=\sum_{k=0}^m\frac{1}{k!}{\sum}'
\frac{(-1)^{|\al|}}{\al^0!\al^1!\cdots\al^k!}(\sigma D_Z/2)^{\al}
(\dif_X^{\al^0}p(X+2Z)\\
\times\dif_X^{\al^1}\phi(X+Z)
\cdots\dif_X^{\al^k}\phi(X+Z))\big|_{Z=0}.
\end{gather*}
Applying the results for $I_j$ and $\tilde I_j$ to $\Sigma (I_j+\tilde I_j)/2$, 
we end the proof of Theorem \ref{thm:matome}.





%

\begin{thebibliography}{99}





\bibitem{BN:arxiv}
{\sc E.Bernardi and T.Nishitani;}  https://doi.org/10.48550/arXiv.2505.21078 {Geometric aspects of hyperbolic operators with double characteristics in presence of type transition.}  




\bibitem{BN:g4}
{\sc E.Bernardi and T.Nishitani;} {\it On the Cauchy problem for noneffectively hyperbolic operators: The Gevrey 4 well-posedness}, 
Kyoto J. Math. {\bf 51} (2011), 767--810.





\bibitem{BN:g3}
{\sc E.Bernardi and T.Nishitani;} {\it  On the Cauchy problem for noneffectively hyperbolic operators. The Gevrey 3 well-posedness}, 
J. Hyperbolic Differ. Equ. {\bf 8} (2011), 615--650.



\bibitem{Bro}
{\sc M.D.Bronshtein;} {\it Cauchy problem for hyperbolic operators with variable multiple characteristics}, Trudy Moscow Math., {\bf 41} (1980), 83-99.

\bibitem{Ho1}
{\sc L.H\"ormander}; {\it The Cauchy problem for differential equations with double characteristics,} 
J. Anal. Math. {\bf 32} (1977), 118-196.


\bibitem{Ho:book3}
{\sc L. H\"ormander}; {The analysis of linear partial differential operators, III},
 Springer, Berlin, 1985.


\bibitem{Ho:Q}
{\sc L.H\"ormander}; {\it Quadratic hyperbolic operators;} in Microlocal Analysis and Applications, ed. by L.Cattabriga, L.Rodino (Springer, Berlin, 1989, pp.118-160.


\bibitem{IP}
{\sc V.Ja.Ivrii and V.M.Petkov}; {\it Necessary conditions for the Cauchy problem for non strictly hyperbolic equations to be well posed,}  Uspehi Mat. Nauk. {\bf 29} (1974), 3-70.

\bibitem{Iv}
{\sc V.Ja.Ivrii;} {\it Correctness of the Cauchy problem in Gevrey classes for non-strictly hyperbolic operators}, Math. USSR Sbornik, {\bf 25} (1975), 365-387.













\bibitem{Ni:KJM}
{\sc T.Nishitani;} {\it Non effectively hyperbolic operators, Hamilton map and bicharacteristics}, J. Math. Kyoto Univ., {\bf 44} (2004), 55-98.





\bibitem{NTa}
{\sc T.Nishitani and M.Tamura}; {\it A class of Fourier integral operators with complex phase related to the Gevrey classes,} J. Pseudo-Differ. Oper. Appl. {\bf 1} (2010), 255-292.

\bibitem{Ni:book}
{\sc T. Nishitani}; {\it Cauchy problem for differential operators with double
  characteristics}, Lecture Notes in Math. {\bf 2202}, 
 Springer, Cham, 2017.
 
 \bibitem{Ni:arxiv}
{\sc T.Nishitani}; https://doi.org/10.48550/arXiv.2208.07534 
 A note on the Cauchy problem for $-D_0^2+2x_1D_0D_2+D_1^2+x_1^3D_2^2+\sum_{j=0}^2b_jD_j$ 
 
 \bibitem{Ni:JPDO}
{\sc T. Nishitani}; {\it  A more direct way to the Cauchy problem for effectively hyperbolic
  operators},
J. Pseudo-Differ. Oper. Appl., {\bf 15} (2024), Article: 20.








\end{thebibliography}
\end{document}